\documentclass[11pt, leqno]{amsart}
\usepackage{graphicx}
\usepackage{amsfonts,delarray,amssymb,amsmath,amsthm,a4,a4wide}
\usepackage{latexsym}
\usepackage{epsfig}
\usepackage{color}
\usepackage[margin=0.97in]{geometry} 
\newtheorem{thm}{Theorem}[section]

\newtheorem{lem}[thm]{Lemma}

\theoremstyle{definition}

\newtheorem{rem}[thm]{Remark}
\numberwithin{equation}{section}
\renewcommand{\div}{\mbox{div}\,}
\newcommand{\trace}{\mbox{trace}\,}
\newcommand{\dist}{\mbox{dist}\,}

\newcommand{\R}{\mathbb R}

\newcommand{\e}{\varepsilon}

\newcommand{\ov}{\overline}
\newcommand{\p}{\partial}

\newcommand{\comment}[1]{}
\def\h{\hspace*{.24in}}

\newenvironment{myindentpar}[1]%
{\begin{list}{}%
         {\setlength{\leftmargin}{#1}}%
         \item[]%
}
{\end{list}}
\begin{document}

\title[Singular Abreu equations and minimizers of convex functionals]{Singular Abreu equations and minimizers of convex functionals  with a convexity constraint }
\author{Nam Q. Le}
\address{Department of Mathematics, Indiana University,
Bloomington, 831 E 3rd St, IN 47405, USA.}
\email{nqle@indiana.edu}
%\date{November 06, 2018}

\thanks{The research of the author was supported in part by the National Science Foundation under grant DMS-1764248}

% ----------------------------------------------------------------
\begin{abstract}
We study the solvability of second boundary value problems of fourth order equations of Abreu type arising from approximation of convex functionals whose Lagrangians depend on the gradient variable, subject to a convexity constraint. 
These functionals arise in different scientific disciplines such as Newton's problem of minimal resistance in physics and monopolist's problem in economics.  The right hand sides of our Abreu type equations are quasilinear expressions of second order; they are highly singular and a priori just measures.
However, our analysis in particular shows that minimizers of the 2D Rochet-Chon\'e model perturbed by a strictly convex lower order term, 
 under a convexity constraint, can be approximated in the uniform norm by solutions of the second boundary value problems of singular Abreu equations.  

\end{abstract}
\maketitle
\section{Introduction }
In this paper, we study the solvability and convergence properties of second boundary value problems of fourth order equations of Abreu type arising from approximation of several convex functionals whose Lagrangians depend on the gradient variable, subject to a convexity constraint. 
Our analysis in particular shows that minimizers of the 2D Rochet-Chon\'e model perturbed by a strictly convex lower order term, 
 under a convexity constraint, can be approximated in the uniform norm by solutions of the second boundary value problems of Abreu type equations.
An intriguing feature of our Abreu type equations is that their right hand sides are quasilinear expressions of second order derivatives of a convex function. As such, they are highly singular and a priori just measures.
The main results consist of Theorems \ref{SBV1}, \ref{SBV2}, \ref{SBV3} and \ref{SBV4} to be precisely stated in Section \ref{main_sec}. In the following paragraphs, we motivate the problems to be studied and recall some previous results in the literature.

Let $\Omega_0$ be a bounded, open, smooth, convex domain in $\R^n$ ($n\geq 2$). Let $F(x, z, p):\R^n\times \R\times \R^n\rightarrow \R$ be a smooth function which is convex in each of the variables $z\in\R$ and $p=(p_1,\cdots, p_n)\in\R^n$.
Let $\varphi$ be a convex and smooth function defined in a neighborhood of $\Omega_0$. In several problems in different scientific disciplines such as Newton's problem of minimal resistance in physics and monopolist's problem in economics (see for example \cite{BFK, CLR1, RC}), one usually encounters the following variational problem with a convexity constraint:
\begin{equation}
\label{minp1}
\inf_{u\in \bar{S}[\varphi,\Omega_0]} \int_{\Omega_0} F(x, u(x), Du(x)) dx
\end{equation}
where
\begin{multline}
\label{barS}
\bar{S}[\varphi, \Omega_0]=\{ u: \Omega_0\rightarrow \R\mid u \text{ is convex, }\\ \text{ u admits a convex extension } \varphi \text{ in a neighborhood of } \Omega_0\}.
\end{multline}

Due to the convexity constraint, it is in general difficult to write down a tractable Euler-Lagrange equation for the minimizers of (\ref{minp1}) \cite{Car2, CLR2, Lions}. 
Lions \cite{Lions} showed that, in the sense of distributions, the Euler-Lagrange equation for a minimizer $u$ of (\ref{minp1}) in a limiting case of the constraint (\ref{barS}) is of the form 
\begin{equation}
\label{EL_Lions}
\frac{\p F}{\p z}(x, u(x), Du(x)) -\sum_{i=1}^n \frac{\p}{\p x_i} \left(\frac{\p F}{\p p_i}(x, u(x), Du(x))\right)=\sum_{i, j=1}^n \frac{\p^2}{\p x_i \p x_j} \mu_{ij}
\end{equation}
for some symmetric non-negative matrix $\mu = (\mu_{ij})_{1\leq i, j\leq n}$ of Radon measures; see also Carlier \cite{Car2} for a new proof of this result and related extensions. 

 The structure of the matrix $\mu$ in (\ref{EL_Lions}), to the best of the author's knowledge, is still mysterious up to now.
Thus, for practical purposes such as implementing numerical schemes to find minimizers of (\ref{minp1}), it is desirable to find suitably explicit approximations of $\mu$ in particular and minimizers of (\ref{minp1}) in general. This has been done by Carlier and Radice \cite{CR} when the Lagrangian $F$ does not depend on the gradient variable $p$; see Sect. \ref{F0_sec} for a quick review.
In this paper, we tackle the more challenging case when  $F$ depends on the gradient variable. 
This case is relevant to many realistic models in physics and economics such as ones described in \cite{BFK, RC}.

\subsection{Fourth order equations of Abreu type approximating convex functionals with a convexity constraint}
\label{F0_sec}
When $\varphi$ is strictly convex in a neighborhood of $\Omega_0$, the Lagrangian $F(x, z, p)$ does not depend on $p$, that is, $F(x, z, p)= F^0 (x, z)$, and uniform convex in its second argument and $\frac{\p }{\p z}F^0(x, z)$ is bounded uniformly in $x$ for each fixed $z$, Carlier and Radice \cite{CR} show that one can approximate the minimizer of 
\begin{equation}
\label{minp2}
\inf_{u\in \bar{S}[\varphi,\Omega_0]} \int_{\Omega_0}  F^0(x, u(x)) dx
\end{equation}
by solutions of second boundary value problems of fourth order equations of Abreu type. More precisely, for each $\e>0$,
consider the following second boundary value problem for a uniformly convex function $u_\e$ on an open Euclidean ball $B$ containing $\overline{\Omega_0}$:
\begin{equation}
\label{crp}
  \left\{ 
  \begin{alignedat}{2}\e\sum_{i, j=1}^{n}U_\e^{ij}  \frac{\p^2 w_\e}{\p x_i \p x_j}~& =g_{\e}(\cdot, u_\e)~&&\text{in} ~B, \\\
 w_\e~&= (\det D^2 u_\e)^{-1}~&&\text{in}~ B,\\\
u_\e ~&=\varphi~&&\text{on}~\p B,\\\
w_\e ~&= \psi~&&\text{on}~\p B,
\end{alignedat}
\right.
\end{equation}
where $\psi:= (\det (D^2\varphi))^{-1}$ on $\p B$,
\begin{equation*}
g_{\e}(x, u) = \left\{\begin{array}{rl}
 \frac{\p F^0}{\p z} (x, u)&  x\in \Omega_0,\\
\frac{1}{\e}(u (x)-\varphi(x) ) & x\in  B\setminus \Omega_0,
\end{array}\right.
\end{equation*}
 and $U_\e= (U_\e^{ij})$ is the cofactor matrix of $D^2 u_\e= \left(\frac{\p^2 u_\e}{\p x_i \p x_j}\right)_{1\leq i, j\leq n}\equiv ((u_\e)_{ij})$ of the uniformly convex function $u_\e$, that is
$$U_\e= (\det D^2 u_\e) (D^2 u_\e)^{-1}.$$
 Carlier and Radice \cite[Theorems 4.2 and 5.3]{CR} show that (\ref{crp}) has a unique uniformly convex solution $u_\e\in W^{4,q}(B)$ (for all  $q<\infty$)  which converges uniformly on $\overline{\Omega_0}$ to the unique minimizer of (\ref{minp2}) when $\e\rightarrow 0$.
 
 The first equation of (\ref{crp}) is a fourth order equation of Abreu type. We will say a few words about this fully nonlinear, geometric equation.
 Let $(U^{ij})= (\det D^2 u) (D^2 u)^{-1}$ and $u_{ij}:= \frac{\p^2 u}{\p x_i \p x_j} $ for any function $u$. The Abreu equation \cite{Ab} for a uniformly convex function $u$
$$\sum_{i, j=1}^n U^{ij}[(\det D^2 u)^{-1}]_{ij}=f$$
first arises in differential geometry \cite{Ab, D1, D2} where one would like to find a K\"ahler metric of constant scalar curvature. 
Its related and important cousin is the affine maximal surface equation \cite{TW00, TW05, TW08} in affine geometry:
$$\sum_{i, j=1}^n U^{ij}[(\det D^2 u)^{-\frac{n+1}{n+2}}]_{ij}=0.$$
 We call (\ref{crp}) the second boundary value problem because the values of the function $u_\e$ and its Hessian determinant $\det D^2 u_\e$ are prescribed on the boundary $\p B$. This is in contrast to the 
 first boundary value problem where one prescribes the values of the function $u_\e$ and its gradient $ D u_\e$ on $\p B$. The fourth order equation in (\ref{crp}) arises as the Euler-Lagrange equation of the functional
\begin{equation*}
\int_{\Omega_0}  F^0(x, u(x)) dx +\frac{1}{2\e}\int_{B\setminus\Omega_0} (u-\varphi)^2 dx-\e\int_{B} \log \det D^2 u dx.
\end{equation*}
At the functional level, the penalization $\e\int_{B} \log \det D^2 u dx $ involving the logarithm of the Hessian determinant acts as a good barrier for the convexity constraint in problems like (\ref{minp2}); see also  \cite{BCMO} for related rigorous numerical results at a discretized level.
At the equation level, the results of Carlier and Radice \cite{CR} show that, when the Lagrangian $F$ does not depend on $p$, the matrix $\mu$ in (\ref{EL_Lions}) is well approximated by $\e (D^2 u_\e)^{-1}\equiv \e (u_\e^{ij})$ where $u_\e$ is the solution of (\ref{crp}). To see this, we just note that the cofactor matrix $U_\e$ of $D^2 u_\e$ is divergence-free, that is $\sum_{j=1}^n \frac{\p}{\p x_j} U_\e^{ij}=0$ for all $i=1,\cdots, n$ and hence, noting that $U^{ij}_\e w_\e= u_\e^{ij}$, we can write the left hand side of the first equation in (\ref{crp})
as
$$\e \sum_{i, j=1}^n U^{ij}_\e \frac{\p^2 w_\e}{\p x_i \p x_j} = \sum_{i, j=1}^n  \frac{\p^2 }{\p x_i \p x_j} (\e U^{ij}_\e w_\e)=  \sum_{i, j=1}^n \frac{\p^2}{\p x_i \p x_j} \e u_\e^{ij}.$$

\subsection{Gradient-dependent Lagrangians and the Rochet-Chon\'e model}
The analysis of Carlier-Radice \cite{CR} left open the question of whether one can approximate minimizers of (\ref{minp1}) by solutions of second boundary value problems of fourth order equations of Abreu type when the Lagrangian $F$ depends on the gradient variable $p$. This case is relevant to physics and economic applications. 
We briefly describe here the  Rochet-Chon\'e model in economics.
In the Rochet-Chon\'e model \cite{RC} of the monopolist problem in product line design where the cost of producing product $q$ is the quadratic function $\frac{1}{2}|q|^2$, the monopolist's profit
as a functional of the buyers' indirect utility function $u$
is
$$\Phi(u)= \int_{\Omega_0} \{x\cdot D u(x) - \frac{1}{2}|D u(x)|^2- u(x)\} \gamma(x) dx.$$
Here $\Omega_0\subset\R^n$ is the collection of  types of agents and $\gamma$ is the relative frequency of different types of agents in the population. For a consumer of type $x\in\Omega_0$, the indirect utility function $u(x)$ is computed via the formula
$$u(x) =\max_{q\in Q} \{x\cdot q-p(q)\}$$
where $Q\subset\R^n$ is the product line and $p: Q\rightarrow \R$ is a price schedule that the monopolist needs to design to maximize her overall profit. Since $u$ is the maximum of a family of affine functions, it is convex. 
Maximizing $\Phi (u)$ over convex functions $u$ is equivalent to minimizing the following functional $J_0$ over convex functions $u$: $$J_0(u) =\int_{\Omega_0}F^{RC}(x, u(x), Du(x)) dx~\text{where } F^{RC}(x, z, p)=\frac{1}{2}|p|^2 \gamma(x) -x\cdot p \gamma(x) + z \gamma (x).$$ 
As mentioned in \cite{Ga}, even in this simple looking variational problem, the convexity is not easy to handle from a numerical standpoint. M\'erigot and Oudet \cite{MO} were among the first to make interesting progress in this direction. Here we analyze this problem, and its generalization, from an asymptotic analysis standpoint.

In this paper, we are interested in using the second boundary value problems of fourth order equations of 
Abreu type to approximate minimizer(s) of
the following variational problem
\begin{equation}
\label{minp3}
\inf_{u\in \bar{S}[\varphi,\Omega_0]} J(u)
\end{equation}
where
\begin{equation} 
\label{J_def}
J(u)=\int_{\Omega_0} F(x, u(x), Du(x)) dx,~ \text{with } F(x, z, p)= F^0(x, z) + F^1(x, p).
\end{equation}
The choice of form of $F$ in (\ref{J_def}) simplifies some of our arguments and is clearly motivated by the analysis of the Rochet-Chon\'e model. 

Similar to the analysis of (\ref{crp}) carried out by Carlier-Radice \cite{CR},
our analysis leads us to two very natural questions concerning the following second boundary value problem of a highly singular, fully nonlinear fourth order equation of Abreu type for a uniformly convex function $u$:
\begin{equation}
\label{Abreu2_log}
  \left\{ 
  \begin{alignedat}{2}\sum_{i, j=1}^{n}U^{ij}w_{ij}~& =f_{\delta}(\cdot, u, Du, D^2 u)~&&\text{in} ~\Omega, \\\
 w~&= (\det D^2 u)^{-1}~&&\text{in}~ \Omega,\\\
u ~&=\varphi~&&\text{on}~\p \Omega,\\\
w ~&= \psi~&&\text{on}~\p \Omega.
\end{alignedat}
\right.
\end{equation}
Here $(U^{ij})_{1\leq i, j\leq n}$ is the cofactor matrix of the Hessian matrix $D^2 u=(u_{ij})$, $\delta>0$, $\Omega$ is a bounded, open, smooth, uniformly convex domain containing $\overline{\Omega_0}$ and
\begin{equation}
\label{fdel_log}
f_{\delta}(x, u(x), Du(x), D^2 u(x)) = \left\{\begin{array}{rl}
  \frac{\p}{\p z}F^0(x, u(x)) - \sum_{i=1}^n \frac{\p}{\p x_i} \left(\frac{\p F^1}{\p p_i}(x, Du(x))\right) &  x\in \Omega_0,\\
\frac{1}{\delta}(u (x)-\varphi(x) ) & x\in  \Omega\setminus \Omega_0.
\end{array}\right.
\end{equation}
{\bf Question 1. } 
\begin{center}
Given $\varphi, \psi, F^0$, and $ F^1$, can we solve the second boundary value problem (\ref{Abreu2_log})-(\ref{fdel_log})?
\end{center}
{\bf Question 2.}
\begin{center}
Are minimizers of (\ref{minp3})-(\ref{J_def}) well approximated by solutions of  (\ref{Abreu2_log})-(\ref{fdel_log}) when $\delta\rightarrow 0$?
\end{center}
We will answer Questions 1 and 2 in the affirmative in two dimensions under suitable conditions on $F^0$ and $F^1$--see Theorems \ref{SBV1} and \ref{SBV2} respectively--via analysis of singular Abreu equations.
 \subsection{Singular Abreu equations}
By now, the second boundary value problem 
for the Abreu equation is well understood \cite{CW, Le, Le_JDE, LMT, TW08}. In particular, from the analysis in \cite{Le_JDE}, we know that
if $f\in L^q(\Omega)$ where $q>n$ then we have a unique uniformly convex $W^{4, q}(\Omega)$ solution to the second boundary value problem  of a more general form of the Abreu equation:
\begin{equation}
\label{Abreu_log}
  \left\{ 
  \begin{alignedat}{2}\sum_{i, j=1}^{n}U^{ij}w_{ij}~& =f~&&\text{in} ~\Omega, \\\
 w~&= G^{'}(\det D^2 u)~&&\text{in}~ \Omega,\\\
u ~&=\varphi~&&\text{on}~\p \Omega,\\\
w ~&= \psi~&&\text{on}~\p \Omega
\end{alignedat}
\right.
\end{equation}
where $\varphi\in W^{4, q}(\Omega)$, $\psi\in W^{2, q}(\Omega)$ with $\inf_{\p\Omega}\psi>0$, and $G$ belongs to a class of concave functions which include $G(t) =\frac{t^{\theta}-1}{\theta}$ where $0< \theta<1/n$ and $G(t) =\log t.$
On the other hand, if the right hand side $f$ is only in $L^q(\Omega)$ with $q<n$ then solutions to (\ref{Abreu_log}) might not be in $W^{4, q}(\Omega)$.

In \cite{CR}, the authors established an a priori uniform bound for solutions of (\ref{crp}) thus confirming the solvability of (\ref{crp}) in all dimensions, where $g_\e$ is now bounded, by using the solvability results for (\ref{Abreu_log}).

The second boundary value problem of  Abreu type in (\ref{Abreu2_log}) has highly singular right hand side even in the simple but nontrivial setting
of $F^0(x, z)=0$ and $F^{1}(x, p)=\frac{1}{2}|p|^2$.
Among the simplest analogue of the first equation in (\ref{Abreu2_log}) is 
\begin{equation}
\label{Abreu_del}
U^{ij}[(\det D^2 u)^{-1}]_{ij}=-\Delta u~\text{in }\Omega.
\end{equation}
To the best of our knowledge, the Abreu type equation of the form (\ref{Abreu_del}) has not appeared before in the literature. 
There are several serious challenges in establishing the solvability of its second boundary value problem. We highlight here two aspects among these challenges:
\begin{myindentpar}{1cm}
(C1) It is not known a priori if we can establish the lower bound and upper bound for $\det D^2 u$. Thus, for a convex function $u$, $\Delta u$ can be only a measure.\\
(C2) Even if  we can establish the positive lower bound $\lambda_1$ and upper bound $\lambda_2$ for $\det D^2 u$, 
that is $\lambda_1\leq\det D^2 u\leq \lambda_2$ in $\Omega$,
we can only deduce from the regularity results for the Monge-Amp\`ere equation of De Philippis-Figalli-Savin \cite{DPFS}, Schmidt \cite{Sch} and Savin \cite{S3} that $\Delta u\in L^{1+\e_0}(\Omega)$ where 
$\e_0=\e_0(n, \lambda_1,\lambda_2)>0$ can be arbitrary small. In fact, from Wang's counter-example \cite{W95} to regularity of the Monge-Amp\`ere equations, we know that $\e_0(n, \lambda_1,\lambda_2)\rightarrow 0$ when $\lambda_2/\lambda_1\rightarrow \infty$. In other words, the right hand side of (\ref{Abreu_del}) has low integrability a priori which can be less than the dimension $n$. Thus, the results on the solvability of the second boundary value problem of the Abreu equation in 
\cite{CW, Le, Le_JDE, LMT, TW08} do not apply to the second boundary value problems of (\ref{Abreu_del}) and (\ref{Abreu2_log}). 
\end{myindentpar}

In this paper, we are able to overcome these difficulties for both (\ref{Abreu2_log}) and (\ref{Abreu_del}) in two dimensions under suitable conditions on the convex functions $F^0$ and $F^1$; see Theorem \ref{SBV1} which asserts the solvability of (\ref{Abreu2_log})-(\ref{fdel_log}). This is done via a priori fourth order derivatives estimates and degree theory. For the a priori estimates, the structural conditions on $F^0$ and $F^1$ allow us to establish that $\lambda_1\leq\det D^2 u\leq \lambda_2$ in $\Omega$ for some 
positive constants $\lambda_1,\lambda_2$ and that $f_\delta(\cdot, u, Du, D^2 u)$, as explained in (C2) for $-\Delta u$, belongs to $L^{1+\e_0}(\Omega)$ for some possibly small $\e_0>0$. 
We briefly explain here how we can go beyond second order derivatives estimates and why the dimension is restricted to 2.

Note that (\ref{Abreu2_log}) consists of a Monge-Amp\`ere equation for $u$ in the form of $\det D^2 u=
w^{-1}$ and a linearized Monge-Amp\`ere equation for $w$ in the form of $U^{ij} w_{ij}=f_\delta(\cdot, u, Du, D^2 u)$ because
the coefficient matrix $(U^{ij})$ comes from linearization of the Monge-Amp\`ere operator: $U=\frac{\p\det D^2 u}{\p u_{ij}}.$ For the solvability of second boundary problems such as (\ref{Abreu2_log}) and (\ref{Abreu_del}), as in \cite{CW, Le, Le_JDE, LMT, TW08}, a key ingredient is to establish global H\"older continuity of the linearized Monge-Amp\`ere equation with right hand side having low integrability. 
In our case, the integrability exponent is $1+\e_0$ for a small $\e_0>0$.

To the best of our knowledge, the lowest integrability exponent $q$ for the right hand side of  the linearized Monge-Amp\`ere equation (with Monge-Amp\`ere measure just bounded away from $0$ and $\infty$) for which one can establish a global H\"older continuity estimate 
is $q>n/2$. This fact was proved in the author's paper with Nguyen \cite{LN}. The constraint $1+\e_0>n/2$ for small $\e_0>0$ forces $n$ to be 2. {\bf It is exactly this reason that we restrict ourselves in this paper to considering the case $n=2$.} Using the bounds on the Hessian determinant for $u$ and  the global H\"older estimates in \cite{LN}, we can show that $w$ is globally H\"older continuous.  
Once we have this,
we can apply the global $C^{2,\alpha}$ estimates for the Monge-Amp\`ere equation in \cite{S, TW08} 
 to conclude that $u\in C^{2,\alpha}(\overline{\Omega})$.  We update this information to $U^{ij} w_{ij}=f_\delta(\cdot, u, Du, D^2 u)$ to have a second order uniformly elliptic  equation for $w$ with global H\"older continuous coefficients
 and bounded right hand side. This gives second order derivatives estimates for $w$. Now, fourth order derivative estimates for $u$ easily follows.

Under suitable conditions on $F^0$ and $F^1$
we can 
show that
solutions to (\ref{Abreu2_log})-(\ref{fdel_log}) converge uniformly on compact subsets of $\Omega$ to the unique minimizer of (\ref{minp3})-(\ref{J_def}); see 
Theorem \ref{SBV2}. It implies in particular that
minimizers of the 2D Rochet-Chon\'e model perturbed by a highly convex lower order term, under a convexity constraint, can be approximated in the uniform norm by solutions of second boundary value problems of singular Abreu equations.

\begin{rem}
Our analysis also covers the case when $w=(\det D^2 u)^{-1}$ in (\ref{Abreu2_log}) is replaced by $w= (\det D^2 u)^{\theta-1}$ where $0\leq \theta<1/n$. We will consider these general cases in our main results.
\end{rem}

\begin{rem}
As mentioned above, due to the possibly low integrability of the right hand side $-\Delta u$ of (\ref{Abreu_del}), our analysis is at the moment restricted to two dimensions. However, for the solvability of the second boundary value problem, it is quite unexpected that the structure of $-\Delta u$ in two dimensions, that is $-\Delta u=-\trace (U^{ij})$, allows us to replace the term
$(\det D^2 u)^{-1}$ in (\ref{Abreu_del}) by $H(\det D^2 u)$ for very general functions $H$ including $H(d)= d^{\theta-1}$ for $\theta\in [0,\infty)\setminus \{1\}$;
 see Theorem \ref{SBV3}.
Note that, it is an open question if (\ref{Abreu_log}) is solvable for $f\in L^q(\Omega)$ when $q>n$ and $G^{'}(d) = d^{\theta-1}$ where $\theta \in [\frac{1}{n},\infty).$
\end{rem}
\begin{rem}
 It should not come as a surprise when Abreu type equations appear in problems motivated
from economics. On the one hand, in addition to \cite{CR} and this paper, Abreu type equations also appear in the continuum Nash's bargaining problem \cite{Warren}.
On the other hand, the monopolist's problem can be treated in the framework of optimal transport (see, for example \cite{FKM, Ga}) so it is not totally unexpected to have deep connections with the Monge-Amp\`ere equation.
The interesting point here is that  Abreu type equations involve both the  Monge-Amp\`ere equation and its linearization.
\end{rem}
{\bf Notation.} The following notations will be used throughout the paper. Points in $\R^n$ will be denoted by $x=(x_1,\cdots, x_n)\in\R^n$ or $p= (p_1,\cdots, p_n)\in\R^n$. $I_n$ is the identity $n\times n$ matrix.
We use  $\nu= (\nu_{1},\cdots,\nu_n)$ to denote  the unit outer normal vector field on $\p \Omega$ and $\nu_0$ on $\p\Omega_0$. Unless otherwise stated, repeated indices are summed 
such as $U^{ij} w_{ij}=\sum_{i, j=1}^n U^{ij} w_{ij};$
$$f^{0}(x, z)=\frac{\p F^0(x, z)}{\p z}; F^{1}_{p_i}(x, p)=\frac{\p F^1(x, p)}{\p p_i}; \nabla_p F^1(x, p)=(F^{1}_{p_1}(x, p),\cdots, F^{1}_{p_n}(x, p));$$
$$F^{1}_{p_i p_j}(x, p)=\frac{\p^2 F^1(x, p)}{\p p_i \p p_j}; F^{1}_{p_i x_j}(x, p)=\frac{\p^2 F^1(x, p)}{\p p_i \p x_j}; \div (\nabla_p F^1(x, p))=\sum_{i=1}^n \frac{\p }{\p x_i} \left(\frac{\p F^1(x, p)}{\p p_i} \right).$$
We use $U=(U^{ij})_{1\leq i, j\leq n}$ to denote the cofactor matrix of the Hessian matrix $D^2 u= \left(\frac{\p^2 u}{\p x_i \p x_j}\right)\equiv (u_{ij})_{1\leq i, j\leq n}$ of a function $u\in C^2 (\overline{\Omega})$. If $u$ is  uniformly convex in $\Omega$ then
$U=(\det D^2 u)(D^2 u)^{-1}.$

The rest of the paper is organized as follows. We state our main results in Section \ref{main_sec}. In Section \ref{Tools_sec}, we recall tools used in the proofs of our main theorems. In Section \ref{Apriori_sec}, we establish a priori estimates. 
The final section \ref{pf_sec} proves the main results in Theorems \ref{SBV1}, \ref{SBV2}
 \ref{SBV3} and \ref{SBV4}.
\section{Statements of the main results}
\label{main_sec}
Let $\delta>0$ and let $\Omega_0,\Omega$ be open, smooth, bounded, convex domains in $\R^n$ such that $\Omega_0\subset\subset \Omega$.

 We study the solvability of the following second boundary value problem of a fully nonlinear, fourth order equation of Abreu type for a uniformly convex function $u$:
\begin{equation}
  \left\{ 
  \begin{alignedat}{2}\sum_{i, j=1}^{n}U^{ij}w_{ij}~& =f_{\delta}(\cdot, u, Du, D^2 u)~&&\text{in} ~\Omega, \\\
 w~&= (\det D^2 u)^{\theta-1}~&&\text{in}~ \Omega,\\\
u ~&=\varphi~&&\text{on}~\p \Omega,\\\
w ~&= \psi~&&\text{on}~\p \Omega.
\end{alignedat}
\right.
\label{Abreu2}
\end{equation}
Here $U=(U^{ij})_{1\leq i, j\leq n}$ is the cofactor matrix of the Hessian matrix $D^2 u=(u_{ij})$ and
\begin{equation}
\label{fdel}
f_{\delta}(x, u(x), Du(x), D^2 u(x)) = \left\{\begin{array}{rl}
  f^0(x, u(x)) - \div (\nabla_p F^1(x, D u(x)))&  x\in \Omega_0,\\
\frac{1}{\delta}(u (x)-\varphi(x) ) & x\in  \Omega\setminus \Omega_0.
\end{array}\right.
\end{equation}

We consider the following sets of assumptions for {\it nonnegative} constants $\rho, c_0, C_\ast, \bar c_0, \bar C_\ast$:
\begin{equation}
\label{AsF0}
(f^0(x, z)-f^0(x, \tilde z))(z-\tilde z)\geq \rho|z-\tilde z|^2; |f^0(x, z)|\leq \eta (|z|)~\text{for all } x\in\Omega_0~\text{and all } z,\tilde z\in\R
\end{equation}
where $\eta:[0,\infty)\rightarrow [0,\infty)$ is a continuous and increasing function. 
\begin{equation}
\label{AsH}
0\leq F^1_{p_i p_j}(x, p)\leq C_\ast I_n; |F^1_{p_i x_i}(x, p)| \leq c_0 |p|+C_\ast \text{ for all }x\in\Omega_0~\text{and for each } i.
\end{equation}
\begin{equation}
\label{AsH1}
|F^1_{p_i}(x, p)| \leq \bar c_0  |p| + \bar C_\ast \text{ for }x\in\p\Omega_0~\text{and for each } i;~|\nabla_p F^1(x, p)|\leq \eta (|p|)~~\text{for all } x\in\Omega_0.
\end{equation}

Our first main theorem is concerned with the solvability of (\ref{Abreu2})-(\ref{fdel}) in two dimensions.
\begin{thm}[Solvability of highly singular second boundary value problems of Abreu type]
\label{SBV1}
Let $n=2$ and $0\leq \theta<1/n$.
Let $\delta>0$ and let $\Omega_0,\Omega$ be open, smooth, bounded, convex domains in $\R^n$ such that $\Omega_0\subset\subset \Omega$. 
Assume moreover that $\Omega$ is uniformly convex. 
Assume that $\varphi\in C^{3,1}(\overline{\Omega})$ and $\psi\in C^{1,1}(\overline{\Omega})$ with $\inf_{\p \Omega}\psi>0$. Assume that (\ref{AsF0})-(\ref{AsH1}) are satisfied. 
\begin{myindentpar}{1cm} 
(i) If either  $\min\{c_0, \bar c_0\}$ is sufficiently small (depending only on $\inf_{\p \Omega}\psi$, $\Omega_0$ and $\Omega$), or $\min\{\rho,\frac{1}{\delta}\}$ is sufficiently large (depending only on $\min\{c_0,\bar c_0\},\Omega_0$ and $\Omega$),  then
there is a uniformly convex solution $u \in W^{4,q}(\Omega)$ to the system (\ref{Abreu2})-(\ref{fdel}) for all $q\in (n,\infty)$.\\
 (ii) if $\bar c_0=\bar C_\ast =0$, then there is a unique uniformly convex solution $u \in W^{4,q}(\Omega)$ to the system (\ref{Abreu2})-(\ref{fdel}) for all $q\in (n,\infty)$.
\end{myindentpar}
\end{thm}
Theorem \ref{SBV1} will be proved in Section \ref{pf_sec}. The existence proof uses a priori estimates in Theorem \ref{thmw4p} and degree theory. For the a priori estimates, the technical size conditions in (i) guarantee the uniform bound for $u$ and the $L^2$ bound for its gradient $Du$ in terms of the data of the problem.
\begin{rem}
\label{pRC_rem}
Consider the perturbed Rochet-Chon\'e model $ F(x, z,p)= F^0(x, z) + F^1(x, p)$ where $$F^0(x, z)=\gamma (x) z + \frac{\rho}{2} |z|^2,~F^1(x, p)=\frac{1}{2}\gamma(x)|p|^2- x\cdot p\gamma(x)$$
where $\rho\geq 0$ is a constant and $\gamma$ is a Lipschitz function satisfying
$0< \gamma\leq C_1, |D \gamma|\leq C_2~\text{in }\Omega_0.$
Then (\ref{AsF0})-(\ref{AsH1}) are satisfied with suitable constants $c_0, \bar c_0, C_\ast,\bar C_\ast$. If $\gamma=0$ on $\p\Omega_0$ then $\bar c_0=\bar C_\ast =0.$ More generally, if $\max_{\p\Omega}\gamma(x)$ is small  then $\bar c_0$ is small. If $\|D \gamma\|_{L^{\infty}(\Omega_0)}$ is small then $c_0$ is small.
\end{rem}
Our second main theorem asserts the convergence of solutions to (\ref{Abreu2})-(\ref{fdel}) in two dimensions  to the unique minimizer of (\ref{minp3})-(\ref{J_def}) when the Lagrangian $F(x, z, p)= F^0(x, z) + F^1(x,p)$ is highly convex in the second variable.

\begin{thm}[Convergence of solutions of the approximate second boundary value problems of Abreu type to the minimizer of the convex functional]
\label{SBV2}
Let $n=2$ and $0\leq \theta<1/n$.
Let $\Omega_0,\Omega$ be open, smooth, bounded, convex domains in $\R^n$ such that $\Omega_0\subset\subset \Omega$. 
Moreover, assume that $\Omega$ is uniformly convex. 
Assume that $\varphi\in C^{3,1}(\overline{\Omega})$ is uniformly convex with $\inf_{\Omega}\det D^2\varphi >0$
and $\psi\in C^{1,1}(\overline{\Omega})$ with $\inf_{\p \Omega}\psi>0$. 
Assume that (\ref{AsF0})-(\ref{AsH1}) are satisfied, and $\rho>0$.  For each $\e>0$,
 consider the following second boundary value problem:
\begin{equation}
  \left\{ 
  \begin{alignedat}{2}\e\sum_{i, j=1}^{n}U_\e^{ij}(w_\e)_{ij}~& =f_{\e}(\cdot, u_\e, Du_\e, D^2 u_\e)~&&\text{in} ~\Omega, \\\
 w_\e~&= (\det D^2 u_\e)^{\theta-1}~&&\text{in}~ \Omega,\\\
u_\e ~&=\varphi~&&\text{on}~\p \Omega,\\\
w_\e ~&= \psi~&&\text{on}~\p \Omega.
\end{alignedat}
\right.
\label{Abreu2e}
\end{equation}
Here $U_\e=(U_\e^{ij})_{1\leq i, j\leq n}$ is the cofactor matrix of the Hessian matrix $D^2 u_\e=((u_\e)_{ij})$ and
\begin{equation}
\label{fdele}
f_{\e}(x, u_\e(x), Du_\e(x), D^2 u_\e(x)) = \left\{\begin{array}{rl}
  f^0(x, u_\e(x)) - \div (\nabla_p F^1(x, D u_\e(x)))&  x\in \Omega_0,\\
\frac{1}{\e}(u_\e (x)-\varphi(x) ) & x\in  \Omega\setminus \Omega_0.
\end{array}\right.
\end{equation}
Assume that either $\bar c_0=\bar C_\ast =0$ or $\rho$ is sufficiently large (depending only on $\bar c_0 +\bar C_\ast$, $\Omega_0$ and $\Omega$). 
Let $u_\e$ be a uniformly convex solution $u_\e \in W^{4,q}(\Omega)$ to the system (\ref{Abreu2e})-(\ref{fdele}) for all $q\in (n,\infty)$.
Then, $u_\e$ converges uniformly on compact subsets of $\Omega$ to the unique minimizer $u\in \bar{S}[\varphi,\Omega_0]$ (defined in (\ref{barS})) of the problem (\ref{minp3}) where $J$ is defined by (\ref{J_def}).
\end{thm}
Theorem \ref{SBV2} will be proved in Section \ref{pf_sec}.
\begin{rem}
For the convergence result in Theorem \ref{SBV2}, we need to establish a uniform bound for $u_\e$ independent of $\e$; see Lemma \ref{lem_uep}. For this, the uniform convexity of $\varphi$ plays an important role.
On the other hand, in Theorem \ref{SBV1}, we basically use the boundary value of $\varphi$ on $\p\Omega$, and therefore $\varphi$ need not be uniformly convex.
\end{rem}
\begin{rem}  Several pertinent remarks on Theorem \ref{SBV2} are in order.
\begin{myindentpar}{1cm}
(i) Theorem \ref{SBV2} is applicable to the perturbed Rochet-Chon\'e model considered in Remark \ref{pRC_rem}. Theorem \ref{SBV2} implies that
minimizers of the 2D Rochet-Chon\'e model perturbed by a highly convex lower order term, under a convexity constraint, can be approximated in the uniform norm by solutions of second boundary value problems of Abreu type equations.\\
(ii) The minimization problem  (\ref{minp3})-(\ref{J_def}) when $F^1(x,p)=\frac{1}{2}\gamma(x) |p|^2-x\cdot p\gamma(x)$ with $\gamma(x)=0$ on $\p\Omega_0$ (that is, $\bar c_0=\bar C_\ast =0$) was studied by Carlier in \cite{Car1}. \\
(iii) In Theorem \ref{SBV2}, when $\rho>0$ is large and $\bar c_0+ \bar C_\ast>0$, we are unable to prove the uniqueness of uniformly convex solutions $u_\e$ to  (\ref{Abreu2e})-(\ref{fdele}). Despite this lack of uniqueness, Theorem 
\ref{SBV2} says that we have the full convergence of all solutions $u_\e$ to the unique minimizer $u\in \bar{S}[\varphi,\Omega_0]$ of the problem (\ref{minp3})-(\ref{J_def}). This is surprising to us.
\end{myindentpar}
\end{rem}

In Theorem \ref{SBV1}, the function $F^1(x, p)$ grows at most quadratically in $p$. The following extension deals with more general Lagrangian $F$.
\begin{thm}
\label{SBV3}
Let $\Omega\subset\R^2$ be an open, smooth, bounded and uniformly convex domain.
Assume that $\varphi\in C^{\infty}(\overline{\Omega})$ and $\psi\in C^{\infty}(\overline{\Omega})$ with $\inf_{\p \Omega}\psi>0$. 
Let $F(p):\R^2\rightarrow \R$  and $H:(0,\infty)\rightarrow (0,\infty)$ be smooth.
Consider the following
second boundary value problem of a fourth order equation of Abreu type for a  uniformly convex function $u$:
\begin{equation}
  \left\{ 
  \begin{alignedat}{2}\sum_{i, j=1}^{2}U^{ij}w_{ij}~& =-\div (\nabla_p F(Du))~&&\text{in} ~\Omega, \\\
 w~&= H(\det D^2 u)~&&\text{in}~ \Omega,\\\
u ~&=\varphi~&&\text{on}~\p \Omega,\\\
w ~&= \psi~&&\text{on}~\p \Omega.
\end{alignedat}
\right.
\label{Abreu3}
\end{equation}
\begin{myindentpar}{1cm}
(i) Assume that $F$ is convex, and that  $F_{p_i p_j} (p)$ is bounded for $p$ bounded. Assume that $H$ is
strictly decreasing, $H(d)\rightarrow 0$ when $d\rightarrow\infty$ and $H(d)\rightarrow \infty$ when $d\rightarrow 0$.
Then there exists a smooth,  uniformly convex solution $u\in C^{\infty}(\overline{\Omega})$ to (\ref{Abreu3}). If $H(d)= d^{\theta-1}$ where $0\leq \theta<1/2$ then the
 solution is unique.\\
(ii) Assume that $0\leq F_{p_i p_j}(p)\leq C_\ast I_2$. Assume that $H$ is strictly monotone and that $H^{-1}$ maps compact subsets of $(0, \infty)$ into compact subsets of $(0,\infty)$.
Then there exists a  smooth,  uniformly convex solution $u\in C^{\infty}(\overline{\Omega})$ to (\ref{Abreu3}).
\end{myindentpar}
\end{thm}
Theorem \ref{SBV3} will be proved in Section \ref{pf_sec}.
\begin{rem} Examples of Lagrangians $F$ in Theorem \ref{SBV3} (i) include
$$F(p)= \frac{1}{s} |p|^{s} ~(s\geq 2, s~\text{integer}),~\text{or}~F(p)= e^{\frac{1}{2}|p|^2}.$$
 \end{rem}
 The existence results in Theorem \ref{SBV1} and \ref{SBV3} can be extended to certain non-convex Lagrangians $F$. To illustrate the scope of our method, we consider the case of Lagrangian
 $$F(x, z, p)=\frac{1}{4}(z^2-1)^2+ \frac{1}{2}|p|^2$$ arising from the study of Allen-Cahn functionals. Our existence result for the singular Abreu equation with Allen-Cahn Lagrangian states as follows.
 \begin{thm}
\label{SBV4}
Let $\Omega\subset\R^2$ be an open, smooth, bounded and uniformly convex domain. 
Assume that $\varphi\in C^{\infty}(\overline{\Omega})$ and $\psi\in C^{\infty}(\overline{\Omega})$ with $\inf_{\p \Omega}\psi>0$.
Then there exists a  smooth,  uniformly convex solution $u\in C^{\infty}(\overline{\Omega})$ to the following
second boundary value problem:
\begin{equation}
\label{Abreu4}
  \left\{ 
  \begin{alignedat}{2}\sum_{i, j=1}^{2}U^{ij}w_{ij}~& =u^3-u-\Delta u~&&\text{in} ~\Omega, \\\
 w~&= (\det D^2 u)^{-1}~&&\text{in}~ \Omega,\\\
u ~&=\varphi~&&\text{on}~\p \Omega,\\\
w ~&= \psi~&&\text{on}~\p \Omega.
\end{alignedat}
\right.
\end{equation}
\end{thm}
Theorem \ref{SBV4} will be proved in Section \ref{pf_sec}.
\begin{rem}
It would be interesting to establish the higher dimensional versions of Theorems \ref{SBV1}, \ref{SBV2}, \ref{SBV3} and \ref{SBV4}.
\end{rem}

\begin{rem}[Universal constants] 
\label{uc_rem}
In Sections \ref{Apriori_sec} and \ref{pf_sec}, we will work with a fixed exponent $q>n$, and
we call a positive constant {\it universal} if  it depends only on $n, \theta, \eta$, $q, \delta, c_0$, $\bar c_0$, $C_\ast$,$\bar C_\ast$, $\rho$, $\Omega$, $\Omega_0$, $\|\varphi\|_{W^{4, q}(\Omega)}$, $\|\psi\|_{W^{2, q}(\Omega)}$ and $\inf_{\p\Omega}\psi$. 
We use $C, C_1, C_2,\cdots,$ to denote universal constants and their values may change from line to line.
\end{rem}

\section{Tools used in the proofs of main theorems}
\label{Tools_sec}
In this section, we recall the statements of two main tools used in the proofs of our main theorems. 

The first tool is the global H\"older estimates for the linearized Monge-Amp\`ere equation with right hand side having low integrability. These estimates were established by Nguyen and the author in \cite[Theorem 1.7]{LN}. They extend in particular the previous result in \cite[Theorem 1.4]{Le}
(see also \cite[Theorem 1.13]{LMT}) where the case of $L^n$ right-hand side was treated.

\begin{thm} [Global H\"older estimates for the linearized Monge-Amp\`ere equation]
\label{global-H}
Let $\Omega$ be a bounded, uniformly convex domain in $\R^{n}$ with $\p\Omega\in C^{3}$. Let $\phi: \overline{\Omega}\rightarrow \R$, $\phi\in C^{0,1}(\overline{\Omega})\cap C^{2}(\Omega)$  be a convex function satisfying
$$0<\lambda\leq \det D^{2}\phi\leq \Lambda<\infty,~
\text{and}~\phi\mid_{\p\Omega}\in C^{3}.$$
Denote by $(\Phi^{ij})= (\det D^2\phi) (D^2\phi)^{-1}$ the cofactor matrix of $D^2 \phi$.
 Let $v\in C(\overline{\Omega})\cap W^{2, n}_{loc}(\Omega)$ be the solution to the linearized Monge-Amp\`ere equation
 \begin{equation*}
 \left\{
 \begin{alignedat}{2}
   \Phi^{ij} v_{ij} ~&=f \h~&&\text{in} ~\Omega, \\\
 v&= \varphi \h~&&\text{on}~ \p\Omega,
 \end{alignedat}
 \right.
\end{equation*}
 where $\varphi\in C^{\alpha}(\p\Omega)$ for some $\alpha\in (0, 1)$ and $f\in L^q(\Omega)$ with $q>n/2$. Then, $v\in C^{\beta}(\overline{\Omega})$ 
with the estimate
$$\|v\|_{C^{\beta}(\overline{\Omega})}\leq C\left(\|\varphi\|_{C^{\alpha}(\p\Omega)} + \|f\|_{L^{q}(\Omega)}\right)$$
where $\beta$ depends only on $\lambda, \Lambda, n, q, \alpha$, and $C$ depends only on $\lambda, \Lambda, n, q, \alpha$, $diam (\Omega)$, $\|\phi\|_{C^3(\p\Omega)}$, $\|\p\Omega\|_{C^3}$ and the uniform convexity of $\Omega.$ 
\end{thm}

The second tool is concerned with the global $W^{2,1+\e_0}$ estimates for the Monge-Amp\`ere equation. They follows from the interior $W^{2, 1+\e_0}$ estimates in De Philippis-Figalli-Savin \cite{DPFS} and Schmidt \cite{Sch} and
the global estimates in 
 Savin \cite{S3} (see also \cite[Theorem 5.3]{Fi}).
\begin{thm}[Global $W^{2, 1+\e_0}$ estimates for the Monge-Amp\`ere equation]
 \label{global-w21} Let $\Omega\subset \R^n$ ($n\geq 2$) be a bounded, uniformly convex domain. 
Let  $\phi:\overline{\Omega}\rightarrow\R$, $\phi \in C^{0,1}(\overline 
\Omega) 
\cap 
C^2(\Omega)$  be a convex function satisfying
\begin{equation*} 
0<\lambda\leq \det D^2 \phi \leq \Lambda \, \mbox{ in }\, \Omega.
\end{equation*}
Assume that $\varphi\mid_{\p\Omega}$ and $\p\Omega$ are of class $C^3$.
Then, there is a positive constant $\e_0\in (0,1)$ depending only $n,\lambda,\Lambda$ and a positive constant $K$ depending only on 
$n, \lambda, \Lambda,\Omega$, $\|\varphi\|_{C^3(\p\Omega)}$ and $\|\p\Omega\|_{C^3}$ such that 
\begin{equation*}
 \|D^2\varphi\|_{L^{ 1+ \e_0}(\Omega)} \leq K.
\end{equation*}
\end{thm}

We will frequently use the following estimates for convex functions.
\begin{lem}
\label{ucon_lem} Let $u$ be a convex function on $\overline{\Omega}$ where $\Omega\subset\R^n$ is an open, bounded, and convex set.
\begin{myindentpar}{1cm}
(i) We have the following $L^{\infty}$ estimate for $u$ in terms of its boundary value and $L^2$ norm: 
\begin{equation*}
\|u\|^2_{L^{\infty}(\Omega)}  \leq C(n,\Omega, \max_{\p \Omega} u) + C(n,\Omega)\int_{\Omega} |u|^2 dx.
\end{equation*}
(ii) If $\Omega_0\subset\subset\Omega$ then
\begin{equation}
\label{Duu}
|Du(x)|\leq \frac{\max_{\p \Omega} u -u(x)}{\dist(x,\p \Omega)} \leq \frac{1}{\dist(\Omega_0,\p\Omega)} (\max_{\p \Omega} u  +  \|u\|_{L^{\infty}(\Omega)})~\text{for any } x\in\Omega_0.
\end{equation}
\end{myindentpar}
\end{lem}
\begin{proof}[Proof of Lemma \ref{ucon_lem}]
(i) The proof is by comparison with cone. We show that if $u\leq 0$ on $\p\Omega$, then
\begin{equation}
\label{uconL1}
\|u\|_{L^{\infty}(\Omega)} \leq \frac{n+1}{|\Omega|}\int_{\Omega} |u| dx.
\end{equation}
Applying this inequality to the convex function $u-\max_{\p\Omega} u$, we obtain
\begin{equation*}
\|u\|^2_{L^{\infty}(\Omega)}  \leq \left(C (n, \Omega) \int_{\Omega} |u|dx + C(n,\Omega, \max_{\p \Omega} u) \right)^2\leq C (n,\Omega) \int_{\Omega} |u|^2 dx + C(n,\Omega, \max_{\p \Omega} u).
\end{equation*}
It remains to prove (\ref{uconL1}) when $u\leq 0$ on $\p\Omega$. Suppose that $|u|$ attains its maximum at $x_0\in\Omega$. Let $ \hat C$ be the cone with base $\p\Omega$
and vertex at $(x_0, u (x_0))$. Then  (\ref{uconL1}) follows from the following estimates
$$\frac{1}{n+1}\| u\|_{L^{\infty}(\Omega)} |\Omega|=\frac{1}{n+1}| u(x_0)||\Omega|=\text{Volume of } \hat C \leq \int_{\Omega}|u| dx.$$
(ii) The estimate (\ref{Duu}) just follows from the convexity of $u$; see, for example \cite[Lemma 3.11]{LMT}.
\end{proof}
\section{A priori estimates for singular Abreu equations}
\label{Apriori_sec}

In this section, $\delta>0$ and $\Omega_0,\Omega$ are open, smooth, bounded, convex domains in $\R^n$ such that $\Omega_0\subset\subset \Omega$. 
We assume moreover that $\Omega$ is uniformly convex.

The main result of this section is the following global a priori estimates for the second boundary value problem (\ref{Abreu2})-(\ref{fdel}).
\begin{thm}
\label{thmw4p}
Let $n=2$,  $0\leq \theta<1/n$, and $q>n$. 
Assume that $\varphi\in W^{4,q}(\Omega)$ and $\psi\in W^{2,q}(\Omega)$
with $\inf_{\p \Omega}\psi>0$. Assume that (\ref{AsF0})-(\ref{AsH1}) are satisfied. Suppose that either  $\min\{c_0, \bar c_0\}$ is sufficiently small (depending only on $\inf_{\p \Omega}\psi$, $\Omega_0$ and $\Omega$), or $\min\{\rho,\frac{1}{\delta}\}$ is sufficiently large (depending only on $\min\{c_0,\bar c_0\},\Omega_0$ and $\Omega$).
Let $u$ be a smooth, uniformly convex solution of the system (\ref{Abreu2})-(\ref{fdel}). 
Then, there is a universal constant $C>0$ such that
$$\|u\|_{W^{4,q}(\Omega)}\leq C.$$
\end{thm}

We refer to Remark \ref{uc_rem} for our convention on {\it universal constants}. We now give the outline of the proof of Theorem \ref{thmw4p}:
\begin{myindentpar}{1cm}
$\bullet$ We first prove the $L^{\infty}$ bound for $u$ (Lemma \ref{ubound})\\
$\bullet$ We next prove the lower bound for the Hessian determinant $\det D^2 u$ and then the upper bound for the Hessian determinant $\det D^2 u$ (Lemma \ref{upwlem})\\
$\bullet$ Finally, we prove the $W^{4,q}$ estimate
\end{myindentpar}
We use $\nu= (\nu_{1},\cdots,\nu_n)$ to denote the unit outer normal vector field on $\p \Omega$ and $\nu_0$ on $\p\Omega_0$.

For simplicity, we introduce the following size condition used in statements of several lemmas:
\begin{center}
({\bf SC}) Either  
$\min\{c_0, \bar c_0\}$ is sufficiently small (depending only on $\inf_{\p \Omega}\psi$, $\Omega_0$ and $\Omega$), or $\min\{\rho,\frac{1}{\delta}\}$ is sufficiently large (depending only on $\min\{c_0,\bar c_0\},\Omega_0$ and $\Omega$).
\end{center}

The following lemma establishes the universal $L^{\infty}$ bound for solutions  to the second boundary value problem (\ref{Abreu2})-(\ref{fdel})
\begin{lem}
\label{ubound}
Let $n\geq 2$, $0\leq \theta<1/n$, and $q>n$.
Let $u$ be a smooth solution of the system (\ref{Abreu2})-(\ref{fdel}). Assume that $\varphi\in W^{4,q}(\Omega)$ and $\psi\in W^{2,q}(\Omega)$
with $\inf_{\p \Omega}\psi>0$.  Assume that (\ref{AsF0})-(\ref{AsH1}) are satisfied. Assume that either $n\geq 3$ or that ({\bf SC}) holds when $n=2$.
 Then, there is a universal constant $C>0$ 
 such that $$\|u\|_{L^{\infty}(\Omega)}  +\int_{\p \Omega}u_{\nu}^n \leq C.$$
\end{lem}

In the proof of Lemma \ref{ubound}, we will use the following basic geometric construction and estimates. 
\begin{lem}\cite[Lemma 2.1 and inequality (2.7)]{Le_JDE}
\label{geoc} 
Let $G: (0,\infty)\rightarrow\R$ be a smooth, strictly increasing and strictly concave function on $(0,\infty)$.  
 Assume that $q>n\geq 2$ and $\varphi\in W^{4,q}(\Omega)$.
 There exist a convex function $\tilde u\in W^{4,q}(\Omega)$ and constants $C$ and $C(G)$ depending only on $n$, $q$, $\Omega$, and $\| \varphi\|_{W^{4, q}(\Omega)}$
 with the following properties:
 \begin{myindentpar}{1cm}
  (i) $\tilde u=\varphi$ on $\p\Omega$,\\
  (ii) $
\| \tilde{u} \|_{C^{3}(\ov{\Omega})} + 
\|\tilde u\|_{W^{4, q}(\Omega)} \leq C,\quad \textrm{and } \det D^2\tilde{u} \ge C^{-1}>0,
$\\
(iii) letting $\tilde{w}=G'(\det D^2 \tilde{u})$, and denoting by $(\tilde{U}^{ij})$ the 
cofactor matrix of $(\tilde{u}_{ij})$, then  $$\left\|\tilde U^{ij}\tilde w_{ij}\right\|_{L^q(\Omega)}\leq C(G),$$
(iv) if $u\in C^2(\overline{\Omega})$ is a convex function with $u=\varphi$ on $\p\Omega$ then for $u_{\nu}^+ =\max (0, u_{\nu})$, we have
\begin{equation*}\|u\|_{L^{\infty}(\Omega)} \leq C + C(n,\Omega)\left(
\int_{\p \Omega}  (u_{\nu}^+)^n \right)^{1/n}.
\end{equation*}
 \end{myindentpar}
 \end{lem}

\begin{proof}[Proof of Lemma \ref{ubound}]  The proof is similar to \cite[Lemma 2.2]{Le_JDE}. Let $u_{\nu}^+ =\max (0, u_{\nu})$. Since $u$ is convex with boundary value $\varphi$ on $\p\Omega$,we have 
 \begin{equation}
 \label{lowerunu}
 u_\nu\geq -\|D\varphi\|_{L^{\infty}(\Omega)}.
 \end{equation}
Our goal is reduced to showing that
\begin{equation}
\label{ubound_red}
\int_{\p \Omega}(u^{+}_{\nu})^n \leq C
\end{equation}
because the universal $L^{\infty}$ bound for $u$ follows from Lemma \ref{geoc} (iv).  

Let $G(t)=\frac{t^{\theta}-1}{\theta}$ for $t>0$ (when $\theta=0$, we set $G(t)=\log t$). Then $G'(t)= t^{\theta-1}$ for all $t>0$ and $w= G^{'}(\det D^2 u)$ in $\Omega$.

Let $\tilde u\in W^{4, q}(\Omega)$ be as in Lemma \ref{geoc}.
The function $\tilde G(d):= G(d^n)$ on $(0,\infty)$ is strictly concave because
$$\tilde G^{''}(d)= n^2 d^{n-2} \left[G^{''}(d^n) d^n + (1-\frac{1}{n})G^{'}(d^n)\right]< 0.$$
Using this, $G'>0$, and the concavity of the map $M\longmapsto (\det M)^{1/n}$ in the space of symmetric matrices $M\geq 0$, we obtain
\begin{eqnarray*}
 \tilde G((\det D^2 \tilde u)^{1/n}) -\tilde G((\det D^2  u)^{1/n})&\leq& \tilde G^{'}((\det D^2 u)^{1/n})((\det D^2 \tilde u)^{1/n}-(\det D^2 u)^{1/n})\\
 &\leq & \tilde G^{'}((\det D^2 u)^{1/n})\frac{1}{n} (\det D^2 u)^{1/n-1} U^{ij} (\tilde u-u)_{ij} .
\end{eqnarray*}
Since $\tilde G^{'}((\det D^2 u)^{1/n}) = n G^{'}(\det D^2 u) (\det D^2 u)^{\frac{n-1}{n}}$, we rewrite the above inequalities as
\begin{equation}
\label{Ga}
G(\det D^2 \tilde u)- G(\det D^2 u)\leq wU^{ij}(\tilde u- u)_{ij}.
\end{equation}
Similarly, for $\tilde w= G^{'}(\det D^2 \tilde u)$, we have
\begin{equation}
\label{Gb}
G(\det D^2 u)- G(\det D^2 \tilde u)\leq \tilde w \tilde U^{ij}(u- \tilde u)_{ij}.
\end{equation}
Adding (\ref{Ga}) and (\ref{Gb}), integrating by parts twice and using the fact that $(U^{ij})$ is divergence free, we obtain
\begin{eqnarray}
0&\geq&  \int_{\Omega} wU^{ij}(u-\tilde u)_{ij} + \tilde w \tilde U^{ij}(\tilde u-u)_{ij}
\nonumber \\
&=&
  \int_{\partial \Omega} w U^{ij} (u_j-\tilde{u}_j ) \nu_i + \int_{\Omega} U^{ij} w_{ij}(u-\tilde{u}) + 
 \int_{\partial \Omega} \tilde w \tilde{U}^{ij} ( \tilde{u}_j-u_j) \nu_i + \int_{\Omega} \tilde U^{ij} \tilde w_{ij} (\tilde{u}-u)\nonumber\\
  &=&\int_{\partial \Omega} (\psi U^{ij}-\tilde w\tilde U^{ij}) (u_j-\tilde{u}_j ) \nu_i + \int_{\Omega} f_\delta(\cdot, u, Du, D^2 u)(u-\tilde{u})+ \int_{\Omega} \tilde U^{ij} \tilde w_{ij} (\tilde{u}-u).
 \label{concave}
\end{eqnarray}
Let us analyze the boundary terms in (\ref{concave}). Since $u-\tilde u=0$ on $\p \Omega$, we have $(u-\tilde u)_j= (u-\tilde u)_{\nu} \nu_j$, and hence
$$U^{ij}(u-\tilde u)_j \nu_i= U^{ij}\nu_j \nu_i (u-\tilde u)_{\nu} = U^{\nu\nu}(u-\tilde u)_{\nu}$$
where
$$U^{\nu\nu} = \det D^2_{x'} u$$
with $x'\perp \nu$ denoting the tangential directions along $\p \Omega$. Therefore, 
\begin{equation} \label{nunu}
(\psi U^{ij}-\tilde w\tilde U^{ij}) (u_j - \tilde{u}_j) \nu_i = (\psi U^{\nu\nu} -\tilde w \tilde U^{\nu\nu})(u_{\nu} - \tilde{u}_{\nu}).
\end{equation}
 {\it To simplify notation, we use $f_\delta$ to denote $f_\delta(\cdot, u, Du, D^2 u)$ when there is no confusion.}\\
 By Lemma \ref{geoc}, the quantities $\tilde u$, $\tilde{u}_{\nu}$, $\tilde{U}^{\nu\nu}$ and $\|\tilde U^{ij}\tilde w_{ij}\|_{L^1(\Omega)} $are universally bounded.
These bounds combined with (\ref{concave}) and (\ref{nunu}) give
\begin{equation}
\label{bdr_cl1}
\int_{\p\Omega} \psi U^{\nu\nu} u_{\nu} 
 \leq C + C\|u\|_{L^{\infty}(\Omega)} + C\int_{\p\Omega} (|U^{\nu\nu}| +|u_\nu|)+\int_{\Omega}-f_\delta (u-\tilde u)dx.
\end{equation}
On the other hand, from $u-\varphi=0$ on $\p \Omega$, we have, with respect to a principle coordinate system at any point $y\in\p \Omega$ (see, e.g., 
\cite[formula (14.95) in \S 14.6]{GT})
$$D_{ij}(u-\varphi)= (u-\varphi)_{\nu}\kappa_i\delta_{ij}, i, j=1, \cdots, n-1,$$
where $\kappa_1,\cdots,\kappa_{n-1}$ denote the principle curvatures of $\p \Omega$ at $y$. 

Let $K=\kappa_1\cdots\kappa_{n-1}$ be the Gauss curvature of $\partial  \Omega$ at
$y\in\p\Omega$. Then, at any $y\in\p \Omega$, by noting that 
$\det D^2_{x'} u =\det (D_{ij}u)_{1\leq i, j\leq n-1}$ and taking the determinants of
\begin{equation}
\label{uij1}
D_{ij} u = u_{\nu}\kappa_i\delta_{ij}-\varphi_{\nu}\kappa_i\delta_{ij} + D_{ij}\varphi,
\end{equation}
we obtain, using also (\ref{lowerunu})
\begin{equation} \label{Gaussc}
U^{\nu\nu}  = K (u_{\nu})^{n-1} + E, \quad \textrm{where } |E| \le C (1+ |u_{\nu}|^{n-2}) = C(1 + (u^{+}_{\nu})^{n-2}).
\end{equation}
Now, it follows from (\ref{bdr_cl1}), (\ref{Gaussc}) and Lemma \ref{geoc}(iv) that
\begin{eqnarray}
\label{concave2}
 \int_{\partial \Omega} K \psi u_{\nu}^n  &\le& C + C\|u\|_{L^{\infty}(\Omega)}+ C \int_{\partial \Omega} (u_{\nu}^{+})^{n-1}  +  \int_{\Omega} -f_\delta (u-\tilde u)dx\nonumber\\
 &\leq&  C + C\left(
\int_{\p \Omega}  (u_{\nu}^+)^n \right)^{(n-1)/n} +  \int_{\Omega} -f_\delta (u-\tilde u)dx.
\end{eqnarray}

We will analyze the last term on the right hand side of (\ref{concave2}).\\
\noindent
By construction, $\|\tilde u\|_{L^{\infty}(\Omega)}\leq C$ and $\|\varphi\|_{L^{\infty}(\Omega)} \leq C$. Thus, we have
\begin{equation}
\label{uvarphi1}
 -\frac{1}{\delta}(u-\varphi) (u-\tilde u) 
  \leq -\frac{|u|^2}{2\delta} + C(\delta)
\end{equation}
where $C(\delta)>0$ is a universal constant.
Using (\ref{AsF0}), we can estimate
\begin{multline}
\label{A_est}
A:=\int_{\Omega_0}- f^0(x, u) (u-\tilde u) dx
\leq   \int_{\Omega_0}- f^0(x, \tilde u) (u-\tilde u) dx-\rho\int_{\Omega_0}|u-\tilde u|^2 dx\\
\leq \int_{\Omega_0} \eta (|\tilde u|)|u-\tilde u| dx-\rho\int_{\Omega_0}|u-\tilde u|^2 dx\leq C\int_{\Omega_0} |u| dx+ C-\rho\int_{\Omega_0}|u-\tilde u|^2 dx.
\end{multline}
\noindent
{\it Step 1: Estimate $\int_{\Omega_0}- f_{\delta} (u-\tilde u) dx$ by expansion.} By the convexity of $u$ and $F^1(x, p)$ in $p$, we have $F^{1}_{p_i p_j} u_{ij}\geq 0$. Moreover, $u\leq \sup_{\p\Omega}\varphi\leq C$ and $|\tilde u|\leq C$. Thus, recalling (\ref{AsH}), we find that
 \begin{equation}
 \label{F1del}
 F^{1}_{p_i p_j} u_{ij} (u-\tilde u) \leq C F^{1}_{p_i p_j} u_{ij} \leq C C_\ast \Delta u.
 \end{equation}
 On the other hand, for any $i=1,\cdots, n$, using (\ref{AsH}) and the first inequality in (\ref{Duu}), we can bound  $F^1_{p_i x_i} (x, Du(x)) (u-\tilde u )$ in $\Omega_0$  by
 \begin{multline}
 \label{F1Du}
 |F^1_{p_i x_i} (x, Du(x)) (u(x)-\tilde u(x) )| \leq (c_0 |Du(x)|+ C_\ast)|u(x)-\tilde u(x)| \\ \leq( c_0 C(\Omega_0,\Omega)
|u(x)| + C) (|u(x)| + C)
\leq c_0 C(\Omega_0, \Omega)|u(x)|^2 + C|u(x)| + C.
\end{multline}
From (\ref{A_est}), (\ref{F1del}) and (\ref{F1Du}) together with the divergence theorem, we find that
\begin{eqnarray}
\label{Step2_exp}
\int_{\Omega_0}- f_{\delta} (u-\tilde u) dx &=& \int_{\Omega_0}[- f^0(x, u(x)) + \div (\nabla_p F^1(x, Du(x)))] (u-\tilde u)dx\nonumber\\
&=& A + \int_{\Omega_0} (F^1_{p_i x_i} (x, Du(x)) + F^1_{p_i p_j} u_{ij}) (u-\tilde u) dx\nonumber\\
&\leq & A + \int_{\Omega_0}  (c_0 C(\Omega_0, \Omega)|u|^2 + C|u| + C) dx+ \int_{\Omega} C_\ast C \Delta u dx \nonumber\\
&\leq& C\|u\|_{L^{\infty}(\Omega)} + C + c_0 C_1(\Omega_0, \Omega)\int_{\Omega_0}|u|^2 dx -\rho\int_{\Omega_0}|u-\tilde u|^2 dx+ C\int_{\p\Omega} (u_\nu)^{+}.
\end{eqnarray}
\noindent
{\it Step 2: Estimate $\int_{\Omega_0}- f_{\delta} (u-\tilde u) dx$ by integration by parts.} 
Note that $\tilde u$ and $|D\tilde u|$ are universally bounded. Thus, using the convexity of $F^1(x, p)$ in $p$ together with (\ref{AsH1}) and (\ref{Duu}), we have the following estimates in $\Omega_0$
\begin{eqnarray}
\label{Step3_IBP1}
 -\nabla_p F^1(x, Du(x))\cdot (Du-D\tilde u) &\leq&  -\nabla_p F^1(x, D\tilde u)\cdot (Du-D\tilde u) \\ &\leq& \eta( |D\tilde u|) (C(\Omega_0,\Omega)|u(x)|+ C) \leq C|u(x)|+ C\nonumber.
\end{eqnarray}
On the other hand, also by (\ref{AsH1}) and (\ref{Duu}), we have the following estimates on $\p\Omega_0$:
\begin{eqnarray}
\label{Step3_IBP2}
(u-\tilde u)\nabla_p F^1 (x, Du(x))\cdot\nu_0 &\leq& (\bar c_0|Du| + \bar C_\ast)|u-\tilde u| \nonumber\\ &\leq&  \bar c_0 C(\Omega_0,\Omega) \|u\|^{2}_{L^{\infty}(\Omega)} + (\bar c_0+\bar C_\ast)C(\Omega_0,\Omega)\|u\|_{L^{\infty}(\Omega)} + C.
\end{eqnarray}
Now, integrating by parts and using (\ref{A_est}) together with (\ref{Step3_IBP1}) -(\ref{Step3_IBP2}), we obtain
\begin{multline}
\label{Step3_IBP3}
\int_{\Omega_0}- f_{\delta}(u-\tilde u)= \int_{\Omega_0}- f^0(x, u(x)) (u-\tilde u) + \int_{\Omega_0}\div (\nabla_p F^1(x, Du(x))) (u-\tilde u)\\
= A + \int_{\p\Omega_0} (u-\tilde u)\nabla_p F^1 (x, Du(x))\cdot\nu_0 + \int_{\Omega_0} -\nabla_p F^1(x, Du(x))\cdot (Du(x)-D\tilde u(x)) \\ 
\leq C + (\bar c_0+\bar C_\ast)C(\Omega_0,\Omega)\|u\|_{L^{\infty}(\Omega)} + C\int_{\Omega_0}|u| dx-\rho\int_{\Omega_0}|u-\tilde u|^2~dx  + \bar c_0 C_2(\Omega_0,\Omega)\|u\|^2_{L^{\infty}(\Omega)}.
\end{multline}
{\it Step 3: $n\geq 3$.} From (\ref{Step3_IBP3}) and (\ref{uvarphi1}), and recalling Lemma \ref{geoc} (iv), we have
\begin{eqnarray*}
\int_{\Omega} -f_\delta (u-\tilde u) = \int_{\Omega\setminus\Omega_0} -\frac{1}{\delta}(u-\varphi)(u-\tilde u)  + \int_{\Omega_0} - f_{\delta}(u-\tilde u)
&\leq& C+ C|u\|_{L^{\infty}(\Omega)} + C\|u\|^2_{L^{\infty}(\Omega)} \\ &\leq& C+ C_2 \left(
\int_{\p \Omega}  (u_{\nu}^+)^n \right)^{2/n} .
\end{eqnarray*}
From $n\geq 3$, (\ref{concave2}), (\ref{lowerunu}) and the above estimates, we obtain
\begin{eqnarray*} 
 \int_{\partial \Omega} K \psi (u^{+}_{\nu})^n  \le C + C\left(\int_{\partial \Omega} (u_{\nu}^{+})^{n}\right)^{(n-1)/n} +  \int_{\Omega} -f_\delta (u-\tilde u)
 \leq C  + C\left(
\int_{\p \Omega}  (u_{\nu}^+)^n \right)^{(n-1)/n}.
\end{eqnarray*} 
From H\"older inequality, $n\geq 3$, and $\inf_{\p\Omega}K\psi>0$, we easily obtain
$\int_{\partial \Omega} (u^{+}_{\nu})^n \le C$ which is (\ref{ubound_red}).

{\it For the rest of the proof of this lemma, we focus on the more difficult case of $n=2$.} We will now use the condition {\bf (SC)}.\\
\noindent
{\it  Step 4: when
$\min\{c_0, \bar c_0\}$ is sufficiently small (depending only on $\inf_{\p \Omega}\psi$, $\Omega_0$ and $\Omega$).}
From (\ref{Step2_exp}) and (\ref{Step3_IBP3}), and recalling (\ref{uvarphi1}),
 we obtain for $\hat c_0=\min\{c_0, \bar c_0\}$
\begin{eqnarray}
\label{n2twocases}
\int_{\Omega} -f_\delta (u-\tilde u)& \leq &  C\|u\|_{L^{\infty}(\Omega)} + C + \hat c_0 C_3(\Omega_0,\Omega)\|u\|^2_{L^{\infty}(\Omega)} +  C\int_{\p\Omega} (u_\nu)^{+} + \int_{\Omega\setminus\Omega_0} -\frac{1}{\delta}(u-\varphi) (u-\tilde u)\nonumber \\
& \leq & C\|u\|_{L^{\infty}(\Omega)} + C + \hat c_0 C_3(\Omega_0,\Omega)\|u\|^2_{L^{\infty}(\Omega)} + C\int_{\p\Omega} (u_\nu)^{+}.
\end{eqnarray}
From (\ref{concave2}), (\ref{n2twocases}) and $n=2$, we deduce from Lemma \ref{geoc} (iv) that
\begin{eqnarray} 
\label{casen2}
 \int_{\partial \Omega} K \psi u_{\nu}^2  &\le& C + C \left(\int_{\partial \Omega} (u_{\nu}^{+} )^2\right)^{1/2} +  \int_{\Omega} -f_\delta (u-\tilde u)\nonumber\\
 &\leq& C \left(\int_{\partial \Omega} (u_{\nu}^{+} )^2\right)^{1/2}  + C + \hat c_0 C_4(\Omega_0,\Omega)\int_{\partial \Omega} (u_{\nu}^{+})^2.
\end{eqnarray} 
Suppose that $\hat c_0$ is small, say  $$\hat c_0 C_4(\Omega_0,\Omega)<(1/2) ( \inf_{\p\Omega}\psi) ( \inf_{\p\Omega}K).$$
Then it follows from (\ref{casen2}), H\"older inequality and $\inf_{\p\Omega} \psi>0$ that
$\int_{\partial \Omega} (u^{+}_{\nu})^2 \le C$. Thus (\ref{ubound_red}) is proved.\\
{\it Step 5: When $\min\{\rho,\frac{1}{\delta}\}$ is sufficiently large (depending only on $\min\{c_0,\bar c_0\},\Omega_0$ and $\Omega$).}\\
{\bf Case 1: $c_0\leq \bar c_0.$} In this case, we use (\ref{Step2_exp}). 
Suppose that $\rho> 4 c_0 C_1(\Omega_0,\Omega)+1>0$. 
Then, 
\begin{equation}
\label{rho_pos}
-\rho\int_{\Omega_0}|u-\tilde u|^2 dx \leq  -\frac{\rho}{2}\int_{\Omega_0}|u|^2 dx + C(\rho)\leq -2c_0 C_1(\Omega_0,\Omega)\int_{\Omega_0}|u|^2 dx + C.
\end{equation}
Combining (\ref{Step2_exp}) and (\ref{rho_pos}) with (\ref{uvarphi1}) and Lemma \ref{geoc}(iv), we get
\begin{multline}
\label{rho_pos2}
\int_{\Omega} -f_\delta (u-\tilde u) =   \int_{\Omega_0}- f_{\delta} (u-\tilde u) + \int_{\Omega\setminus\Omega_0} -\frac{1}{\delta}(u-\varphi) (u-\tilde u)
\leq C+ C \left(
\int_{\p \Omega}  (u_{\nu}^+)^n \right)^{1/n}.
\end{multline}
Thus, for $n=2$, we deduce from (\ref{concave2}) and (\ref{rho_pos2}) that
\begin{eqnarray}
\label{rho_pos3}
\int_{\partial \Omega} K \psi u_{\nu}^2  \leq C + C \left(
\int_{\p \Omega}  (u_{\nu}^+)^2 \right)^{1/2} +  \int_{\Omega} -f_\delta (u-\tilde u)
\leq C+ C \left(
\int_{\p \Omega}  (u_{\nu}^+)^2\right)^{1/2}.
\end{eqnarray}
From (\ref{rho_pos3}), the H\"older inequality and $\inf_{\p\Omega} \psi>0$, we easily obtain
$\int_{\partial \Omega} (u^{+}_{\nu})^n \le C$ and (\ref{ubound_red}) follows.\\
{\bf Case 2: $ \bar c_0\leq c_0$.}  From  (\ref{Step3_IBP3}), Lemma \ref{ucon_lem} (i), and recalling (\ref{uvarphi1}), we obtain 
\begin{multline}
\label{barc0}
\int_{\Omega} -f_\delta (u-\tilde u)\leq   C\|u\|_{L^{\infty}(\Omega)} + C + \bar c_0 C_3(\Omega_0,\Omega)\|u\|^2_{L^{\infty}(\Omega)} \\ -\rho\int_{\Omega_0}|u-\tilde u|^2 dx+ \int_{\Omega\setminus\Omega_0} -\frac{1}{\delta}(u-\varphi) (u-\tilde u) \\
\leq  C\|u\|_{L^{\infty}(\Omega)} + C + \bar c_0 C_4 (\Omega_0,\Omega)\int_{\Omega} |u|^2 dx  -\frac{\rho}{2}\int_{\Omega_0}|u|^2 dx+ \int_{\Omega\setminus\Omega_0} -\frac{|u|^2}{2\delta}.
\end{multline}
Thus, if $\min\{\rho, \frac{1}{\delta}\}>4 \bar c_0 C_4 (\Omega_0,\Omega)$, then (\ref{barc0}) gives
\begin{equation}
\label{barc02}
\int_{\Omega} -f_\delta (u-\tilde u)\leq   C\|u\|_{L^{\infty}(\Omega)} + C.
\end{equation} 
Now, for $n=2$, we deduce from (\ref{concave2}), (\ref{barc02}) and Lemma \ref{geoc}(iv) that
\begin{equation}
\label{barc03}
\int_{\partial \Omega} K \psi u_{\nu}^2  \leq C + C \left(
\int_{\p \Omega}  (u_{\nu}^+)^2 \right)^{1)/2} +  \int_{\Omega} -f_\delta (u-\tilde u)
\leq C+ C \left(
\int_{\p \Omega}  (u_{\nu}^+)^2 \right)^{1/2}.
\end{equation}
From (\ref{barc03}), the H\"older inequality and $\inf_{\p\Omega} \psi>0$, we easily obtain
$\int_{\partial \Omega} (u^{+}_{\nu})^2 \le C$. 
This completes the proof of (\ref{ubound_red}) in all cases.
\end{proof}

In the following lemma, we establish universal a priori estimates for solutions to (\ref{Abreu2e})-(\ref{fdele}). 
These estimates do not depend on $\e$.
\begin{lem}
\label{lem_uep}
Let  $n=2$, $q>n$ and
$0\leq \theta<1/n$.
Assume that (\ref{AsF0})-(\ref{AsH1})  are satisfied, and $\rho>0$.  
Assume that $\varphi\in W^{4,q}(\Omega)$ with $\inf_{\Omega}\det D^2\varphi>0$, and $\psi\in W^{2,q}(\Omega)$
with $\inf_{\p \Omega}\psi>0$. Assume that one of the following conditions holds:
\begin{myindentpar}{1cm}
(i) $\bar c_0=\bar C_\ast =0$.\\
(ii)  $\rho$ is large and $\e$ is small ( depending only on $\bar c_0 +\bar C_\ast$, $\Omega_0$ and $\Omega$).
\end{myindentpar}
Let $u_\e \in W^{4,q}(\Omega)$ be a uniformly convex solution to the system (\ref{Abreu2e})-(\ref{fdele}).  
Then, there is a universal constant $C>0$ (depending also on  $\inf_{\Omega}\det D^2\varphi$ but independent of $\e$) such that
\begin{eqnarray} 
\label{uep_ineq}
 \int_{\partial \Omega} \e  (u_\e)_{\nu}^2  + \rho \int_{\Omega_0}|u_\e-\varphi|^2 dx +    \int_{\Omega\setminus\Omega_0} \frac{1}{\e}|u_\e-\varphi|^2 dx \leq C.
\end{eqnarray} 
\end{lem}
\begin{proof}
Let $ \hat C_\ast:=\bar c_0+ \bar C_\ast$. Let $\bar u=\varphi$ and $\bar f=  \bar U^{ij} \bar w_{ij}$ where $\bar w=(\det D^2 \bar u)^{\theta-1}$ and $\bar U=(\bar U^{ij})$ is the cofactor matrix of $D^2\bar u$. Since $q>n$, and $\bar u\in W^{4, q}(\Omega)$ with $\inf_{\Omega}\det D^2 \bar u>0$, we have
\begin{equation}
\label{barfL1}
\|\bar f\|_{L^1(\Omega)}\leq C.
\end{equation}
We redo the estimates in Lemma \ref{ubound} for $u_\e$ where we replace $\tilde u$ by $\bar u$.
First, 
(\ref{concave}) becomes
\begin{equation}
\label{concave2e}
 \int_{\partial \Omega} w_\e U_\e^{ij} ((u_\e)_j-\bar{u}_j ) \nu_i + \int_{\Omega} U_\e^{ij} (w_\e)_{ij}(u_\e-\bar{u}) + 
 \int_{\partial \Omega} \bar w \bar{U}^{ij} ( \bar{u}_j-(u_\e)_j) \nu_i + \int_{\Omega} \bar U^{ij} \bar w_{ij} (\bar{u}-u_\e)\leq 0.
\end{equation}
 To simplify notation, we use $f_{\e}$ to denote $f_{\e}(\cdot, u_\e, Du_\e, D^2 u_\e)$.
 
From $U_\e^{ij} (w_{\e})_{ij}=\e^{-1} f_{\e}$, Lemma \ref{geoc}(iv), (\ref{barfL1}), and the uniform boundedness of $\bar u, \bar w, D\bar u$, 
(\ref{concave2e}) becomes 
\begin{eqnarray}
\label{concave2p}
 \int_{\partial \Omega}  (u_\e)_{\nu}^2  &\le&  C \left(\int_{\partial \Omega} (u_\e)_{\nu}^{2}\right)^{1/2}   +  \int_{\Omega} -\e^{-1}f_{\e} (u_\e-\bar u) dx + C.
\end{eqnarray}
This estimate is similar to (\ref{concave2}) where now we also use $\inf_{\p \Omega}\psi>0$ and the uniform convexity of $\p\Omega$ to absorb $\inf_{\p \Omega}\psi>0$ and the curvature  of $\p\Omega$ to the right hand side
of (\ref{concave2}). From Young's inequality, we can absorb the first  term on the right hand side of (\ref{concave2p}) to its left hand side.  Then, multiplying both sides by $\e$, we get
\begin{eqnarray}
\label{concave3p}
 \int_{\partial \Omega} \e  (u_\e)_{\nu}^2  &\le& C +C \int_{\Omega} -f_{\e} (u_\e-\bar u).
\end{eqnarray}
As in the estimates for $A$ in (\ref{A_est}), we have
\begin{eqnarray*}
A_\e:=\int_{\Omega_0}- f^0(x, u_\e(x)) (u_\e-\bar u) dx \leq  C +  C\int_{\Omega_0}|u_\e| dx -\rho\int_{\Omega_0}|u_\e-\bar u|^2 dx.
\end{eqnarray*}
From Lemma \ref{ucon_lem} (i)  and the inequality $|u_\e|^2 \leq 2(|u_\e-\varphi|^2 +|\varphi|^2)$, we have 
\begin{equation}
\label{ue_max2}
\|u_\e\|^2_{L^{\infty}(\Omega)} \leq C (\Omega) \int_{\Omega} |u_\e|^2 dx + C \leq  C (\Omega) \int_{\Omega} |u_\e-\varphi|^2 dx + C.
\end{equation}
Using (\ref{Step3_IBP3}) to $u_\e, f_\e, \bar u$ and taking into account (\ref{ue_max2}), $\rho>0$ and  $ \hat C_\ast:=\bar c_0+ \bar C_\ast$, we have
\begin{eqnarray}
\label{ubar0}
\int_{\Omega_0}- f_{\e} (u_\e-\bar u)&\leq& C + C \int_{\Omega_0}|u_\e| dx  -\rho\int_{\Omega_0}|u_\e-\bar u|^2 dx +  \hat C_\ast C_5(\Omega_0,\Omega) \left(1+ \|u_\e\|^{2}_{L^{\infty}(\Omega)}\right)  \nonumber\\ &\leq&C  -\frac{\rho}{2}\int_{\Omega_0}|u_\e-\bar u|^2 dx +  \hat C_\ast
C_6(\Omega_0,\Omega) \int_{\Omega} |u_\e-\varphi|^2 dx.
\end{eqnarray}
It follows from (\ref{concave3p}), (\ref{ubar0}), $\bar u=\varphi$ in $\Omega$, and $f_\e=\frac{1}{\e} (u_\e-\varphi)$ on $\Omega\setminus\Omega_0$ that
\begin{eqnarray} 
\label{ep_apriori}
\int_{\partial \Omega} \e  (u_\e)_{\nu}^2  &\le& C  +  \int_{\Omega} -f_{\e} (u_\e-\bar u) dx= C + \int_{\Omega_0}- f_{\e} (u_\e-\bar u)+ \int_{\Omega\setminus\Omega_0} -f_{\e} (u_\e-\bar u) dx\nonumber\\
 &\leq& C   -\frac{\rho}{2}\int_{\Omega_0}|u_\e-\varphi|^2 dx +  \int_{\Omega\setminus\Omega_0} -\frac{1}{\e}(u_\e-\varphi)^2 dx +  \hat C_\ast C_6(\Omega_0,\Omega)  \int_{\Omega} |u_\e-\varphi|^2 dx.
\end{eqnarray} 
{\it Case 1: $\rho>0$ and $\bar c_0=\bar C_\ast =0$.} In this case, $\hat C_\ast=0$ in (\ref{ep_apriori}) and (\ref{uep_ineq}) follows from this inequality.\\
\noindent
{\it Case 2: $\rho$ is sufficient large and $\e$ is sufficiently small (depending only on $\bar c_0 +\bar C_\ast$, $\Omega_0$ and $\Omega$).} When $$\min\{\frac{1}{\e}, \rho\}> 4\hat C_\ast C_6(\Omega_0,\Omega)+2,$$ then clearly (\ref{ep_apriori}) gives (\ref{uep_ineq}).
\end{proof}
We prove the uniqueness part of Theorem \ref{SBV1} in the following lemma.
\begin{lem}
\label{uni_lem}
Let $n\geq 2$, $q>n$ and $0\leq \theta<1/n$.
Assume that $\varphi\in W^{4,q}(\Omega)$ and $\psi\in W^{2,q}(\Omega)$ with $\inf_{\p \Omega}\psi>0$. Assume that (\ref{AsF0})-(\ref{AsH1})  are satisfied with $\bar c_0=\bar C_\ast =0$. 
Then the problem (\ref{Abreu2})-(\ref{fdel}) has at most one uniformly convex solution $u\in W^{4,q}(\Omega)$.
\end{lem}
\begin{proof} Suppose that $u\in W^{4,q}(\Omega)$ and $\hat u\in W^{4,q}(\Omega)$ are two uniformly convex solutions of (\ref{Abreu2})-(\ref{fdel}). 
Let $\hat U= (\hat U^{ij})$ be the cofactor matrix of $D^2 \hat u$ and let $\hat w= G^{'}(\det D^2 \hat u)$. Here, $G(t)=\frac{t^{\theta}-1}{\theta}$ for $t>0$ (when $\theta=0$, we set $G(t)=\log t$).
We use the same notation as in the proof of Lemma \ref{ubound}. Then, we obtain as in (\ref{concave}) the estimate
\begin{eqnarray} 
\label{uni_ineq}
0&\geq& 
  \int_{\Omega} wU^{ij}(u-\hat u)_{ij} + \hat w \hat U^{ij}(\hat u-u)_{ij}\nonumber\\
  &=& \int_{\partial \Omega} \psi (U^{ij}-\hat{U}^{ij}) (u_j -\hat{u}_j ) \nu_i +\int_{\Omega} (f_{\delta}(\cdot, u, Du, D^2 u)- f_{\delta}(\cdot, \hat u, D\hat u, D^2\hat u)) (u-\hat u) dx\nonumber \\ &=&
 \int_{\partial\Omega} \psi (U^{\nu\nu}-\hat U^{\nu\nu})(u_{\nu}-\hat u_{\nu}) + \int_{\Omega_0} [f^0(x, u(x))- f^0(x, \hat u(x))] (u-\hat u) dx \nonumber \\&+& \int_{\Omega_0}\div (\nabla_p F^1(x, D\hat u(x))-\nabla_p F^1(x, Du(x))) (u-\hat u) dx + \frac{1}{\delta}\int_{\Omega\setminus \Omega_0} (u-\hat u)^2 dx.
\end{eqnarray}
By (\ref{AsF0}), the integral concerning $f^0$ in the above expression is nonnegative. Hence, integrating by parts and using $\bar c_0=\bar C_\ast =0$, we find that
\begin{equation}
 \label{Gauss3}
0\geq \int_{\partial\Omega} \psi (U^{\nu\nu}-\hat U^{\nu\nu})(u_{\nu}-\hat u_{\nu}) +   \int_{\Omega_0}(\nabla_p F^1(x, D\hat u(x))-\nabla_p F^1(x, Du(x))) (D \hat u-D u) dx.
\end{equation}
It is clear from (\ref{uij1}) that if $u_\nu>\hat u_{\nu}$ then $U^{\nu\nu}>\hat U^{\nu\nu}$. Therefore, from the convexity of $F^1(x, p)$ in $p$ which follows from (\ref{AsH}), we deduce that $u_\nu=\hat u_\nu$ on $\p\Omega$ and that 
(\ref{Gauss3}) must be an equality. This implies that (\ref{uni_ineq}) must be an equality. 
From the derivation of (\ref{uni_ineq}) using the strict concavity of $G$ as in (\ref{Ga}) and (\ref{Gb}) but applied to $\det D^2 u$ and $\det D^2 \hat u$, 
and the fact that (\ref{uni_ineq}) is now an equality, we deduce that  $\det D^2 u=\det D^2 \hat u$ in $\Omega$. Hence $u=\hat u$ on $\overline{\Omega}$.
\end{proof}

In the next lemma, we establish the universal bounds from below and above for the Hessian determinant of solutions to (\ref{Abreu2})-(\ref{fdel}).

\begin{lem}  \label{upwlem} 
Let $n=2$, $q>n$ and $0\leq \theta<1/n$.
Assume that $\varphi\in W^{4,q}(\Omega)$ and $\psi\in W^{2,q}(\Omega)$
with $\inf_{\p \Omega}\psi>0$. Assume that (\ref{AsF0})-(\ref{AsH1}) are satisfied. Suppose that ({\bf SC}) holds.
Let $u$ be a smooth, uniformly convex solution of the system (\ref{Abreu2})-(\ref{fdel}). 
There is a universal constant $C>0$ such that 
$$C^{-1}\leq \det D^2 u\leq C ~\text{in }\Omega.$$
\end{lem}
\begin{proof}[Proof of Lemma \ref{upwlem}]  To simplify notation, we use $f_\delta$ to denote $f_\delta(\cdot, u, Du, D^2 u)$.\\
{\it Step 1: Lower bound for $\det D^2 u$.} Let $\bar w=w+ C_\ast |x|^2$. Let $\chi_{\Omega_0}$ be the characteristic function of $\Omega_0$, that is $\chi_{\Omega_0}(x)=1$ if $x\in\Omega_0$ and
$\chi_{\Omega_0}(x)=0$, if otherwise. Then, using (\ref{fdel}), (\ref{AsH}) and the universal bound for $u$ in Lemma \ref{ubound}, we find that in $\Omega$ the following hold:
\begin{eqnarray*}U^{ij}\bar w_{ij} = f_{\delta} + 2C_\ast\Delta u& \geq& -C - |F^1_{p_i x_i}(x, Du)|\chi_{\Omega_0}(x)- F^1_{p_i p_j}(x, Du) u_{ij} + 2C_\ast\Delta u \\ &\geq&  -C -C_\ast|D u|\chi_{\Omega_0}:=-\hat f.
\end{eqnarray*}
By  combining the universal bound for $u$ in Lemma \ref{ubound} and Lemma \ref{ucon_lem}(ii), we obtain that $\|\hat f\|_{L^{\infty}(\Omega)}\leq C$ for a universal constant $C$. Thus, $\|\hat f\|_{L^{2}(\Omega)}\leq C$. In two dimensions, we have $$\det (U^{ij})=(\det D^2 u)^{n-1}=\det D^2 u=w^{\frac{1}{\theta-1}}$$ where we used the second equation of (\ref{Abreu2}) 
for the last equality. 
Now, we apply the ABP estimate \cite[Theorem 2.21]{HL} to $\bar w$ on $\Omega$ to obtain 
\begin{equation*}
\|\bar w\|_{L^{\infty}(\Omega)} 
 \leq \sup_{\p\Omega} \bar w + C_2 \text{diam}(\Omega) \|\frac{\hat f}{(\det U^{ij})^{1/2}}\|_{L^{2}(\Omega)} \leq  \sup_{\p\Omega} (\psi + C_\ast |x|^2) + C  \|w^{\frac{1}{2(1-\theta)}}\hat f\|_{L^{2}(\Omega)}. 
\end{equation*}
Clearly, $\bar w\geq w>0$.  Therefore, the above estimates and $0\leq \theta<1/2$ give
$$\|w\|_{L^{\infty}(\Omega)}  \leq C + C  \|w^{\frac{1}{2(1-\theta)}}\hat f\|_{L^{2}(\Omega)}\leq C + C \|w^{\frac{1}{2(1-\theta)}}\|_{L^{\infty}(\Omega)} \|\hat f\|_{L^{2}(\Omega)} \leq C + C \|w\|^{\frac{1}{2(1-\theta)}}_{L^{\infty}(\Omega)}.$$
Because $0\leq \theta<1/2$, we have $\frac{1}{2(1-\theta)}<1$. It follows that $w$ is bounded from above.
Since $\det D^2 u=w^{\frac{1}{\theta-1}}$, 
we conclude that $\det D^2 u$ is bounded from below by a universal constant $C>0$.\\
{\it Step 2: Upper bound for $\det D^2 u$.} Since $\det D^2 u=w^{\frac{1}{\theta-1}}$ where $0\leq \theta<1/2$, to prove the universal upper bound for $\det D^2 u$, we only need to obtain a positive lower bound for $w$. 
For this, we use the ABP maximum principle; see \cite{CW,TW05} for a slightly different argument. 

First, we will use the following splitting of $f_\delta$: 
\begin{equation}
\label{wgamma1}
U^{ij} w_{ij}=f_\delta=\gamma_1 \Delta u + g~\text{in }\Omega,
\end{equation}
where
$$\gamma_1(x) = \left\{\begin{array}{rl}
 - \frac{F^1_{p_i p_j} (x, Du) u_{ij}}{\Delta u} &  x\in \Omega_0,\\
0 & x\in  \Omega\setminus \Omega_0
\end{array}\right.$$
and
\begin{equation}
\label{g_formu}
g(x) = \left\{\begin{array}{rl}
 f^0(x, u(x)) - F^1_{p_i x_i}(x, Du(x)) &  x\in \Omega_0,\\
\frac{1}{\delta}(u (x)-\varphi(x) ) & x\in  \Omega\setminus \Omega_0.
\end{array}\right.
\end{equation}
From (\ref{AsH}), we have
\begin{equation}
\label{gamma1}
\|\gamma_1\|_{L^{\infty}(\Omega)}\leq C_\ast.
\end{equation}
 We claim that
 \begin{equation}
 \label{gL2}
 \|g\|_{L^{\infty}(\Omega)}\leq C.
 \end{equation}
 Indeed, from (\ref{AsF0}) and (\ref{AsH}), we easily find that for all $x\in\Omega$
 \begin{equation}
   \label{gptwise}
 |g(x)| \leq \left\{\begin{array}{rl}
 \eta (\|u\|_{L^{\infty}(\Omega)}) + c_0 |Du(x)| + C_\ast &  x\in \Omega_0,\\
 \frac{1}{\delta} (\|u\|_{L^{\infty}(\Omega)} + \|\varphi\|_{L^{\infty}(\Omega)})& x\in  \Omega\setminus \Omega_0.
\end{array}\right.
 \end{equation}
Now, we use the universal bound for $u$ in Lemma \ref{ubound} and Lemma \ref{ucon_lem} (ii) to derive (\ref{gL2}) from (\ref{gptwise}).
 
Recall from (i) that $\det D^2 u\geq C_1$.
Thus, $$w \det D^2 u=(\det D^2 u)^{\theta}\geq C_1^{\theta}:=c>0.$$

From (\ref{wgamma1}) and $\gamma_1\leq 0$, we find that $f_\delta\leq g$. 
By (\ref{gL2}), we find that $f_\delta^{+}\leq |g|$ 
is bounded by a universal constant. 
Let $$M=\frac{ |f_{\delta}^{+}|_{L^{\infty}(\Omega)} + 1}{2 c}<\infty~\text{and } v^\e= \log (w+\e)-Mu\in W^{2,q}(\Omega)$$
where  $\e>0$. Then, in $\Omega$, we have
\begin{eqnarray*}u^{ij} v^\e_{ij} = u^{ij} \left(\frac{w_{ij}}{w+\e}-\frac{w_i w_j}{(w+\e)^2} -M u_{ij}\right)&\leq& \frac{u^{ij} w_{ij} }{w+\e} - nM = \frac{f_{ \delta}}{(w+\e) \det D^2 u}-2 M \\ &\leq&  \frac{\|f^{+}_{ \delta}\|_{L^{\infty}(\Omega)}}{c}-2 M  <0.
\end{eqnarray*}
By the ABP estimate (see \cite[Theorem 2.21]{HL}) for $-v^\e$ in $\Omega$, we have
$$v^{\e}\geq \inf_{\p\Omega} v^\e\geq \log (\inf_{\p\Omega}\psi) -M\varphi\geq -C~\text{in }\Omega.$$
From $v^\e= \log (w+\e)-Mu$ and the universal bound for $u$ in Lemma \ref{ubound}, we obtain  $\log (w+ \e)\geq -C$. Thus, letting $\e\rightarrow 0$, we get $w\geq e^{-C}$ as desired.
\end{proof}

Now, we are a in position to prove Theorem \ref{thmw4p}.

\begin{proof}[Proof of Theorem \ref{thmw4p}]  To simplify notation, we use $f_\delta$ to denote $f_\delta(\cdot, u, Du, D^2 u)$.\\
From Lemma \ref{upwlem}, we can find a universal constant $C>0$ such that
\begin{equation}
\label{udet_bound}
C^{-1}\leq\det D^2 u\leq C~\text{in }\Omega.
\end{equation}
From $\varphi\in W^{4, q}(\Omega)$ with $q>n$, we have $\varphi\in C^3(\overline{\Omega})$ by the Sobolev embedding theorem. By assumption, $\Omega$ is bounded, smooth and uniformly convex.
From $u=\varphi$ on $\p\Omega$ and (\ref{udet_bound}), we can apply the global $W^{2, 1+\e_0}$ estimates for the Monge-Amp\`ere equation in Theorem \ref{global-w21} to conclude that
\begin{equation}
\label{D^2uL1}
\|D^2 u\|_{L^{1+\e_0}(\Omega)}\leq C_1
\end{equation}
for some universal constants $\e_0>0$ and $C_1>0$. 
Recall the following splitting of $f_\delta$ in (\ref{wgamma1}):
\begin{equation}
\label{fdeltag1}
f_\delta= \gamma_1 \Delta u + g.
\end{equation}
 Thus, from (\ref{D^2uL1}), (\ref{gamma1}) and (\ref{gL2}), we find that
  $$\|f_\delta \|_{L^{1+\e_0}(\Omega)}\leq C_2$$ for a universal constant $C_2>0$. 
  From $\psi\in W^{2, q}(\Omega)$ with $q>n$, we have $\psi\in C^1(\overline{\Omega})$ by the Sobolev embedding theorem. 
  Now, we apply the global H\"older estimates for the linearized Monge-Amp\`ere equation in Theorem \ref{global-H} to $U^{ij} w_{ij}=f_{\delta}$ in $\Omega$ with boundary value $w=\psi\in C^{1}(\p\Omega)$ on $\p\Omega$ to 
  conclude that $w\in C^{\alpha}(\overline{\Omega})$
  with 
  \begin{equation}
  \label{walpha}
  \|w\|_{C^{\alpha}(\overline{\Omega})}\leq  C\left(\|\psi\|_{C^{1}(\p\Omega)} + \|f_\delta\|_{L^{1+\e_0}(\Omega)}\right)\leq C_3
  \end{equation}
  for universal constants $\alpha\in (0, 1)$ and $C_3>0$. Now, we note that $u$ solves the Monge-Amp\`ere equation
  $$\det D^2 u= w^{\frac{1}{\theta-1}}$$
  with right hand side being in $C^{\alpha}(\overline{\Omega})$ and boundary value $\varphi\in C^3(\p\Omega)$ on $\p\Omega$.
  Therefore, by the global $C^{2,\alpha}$ estimates for the Monge-Amp\`ere equation \cite{TW08, S}, we have $u\in C^{2,\alpha}(\overline{\Omega})$
  with universal estimates
  \begin{equation}
  \label{u2alpha}
   \|u\|_{C^{2,\alpha}(\overline{\Omega})}\leq C_4~\text{and } C_4^{-1} I_2\leq D^2 u\leq C_4 I_2.
   \end{equation}
   As a consequence, the second order operator $U^{ij} \p_{ij}$ is uniformly elliptic with H\"older continuous coefficients. 
   
Recalling (\ref{fdeltag1}), and using (\ref{gamma1}) together with (\ref{gL2}), we obtain
\begin{equation}
\|f_\delta \|_{L^{\infty}(\Omega)}\leq C_5.
\end{equation}
Thus, from the equation $U^{ij} w_{ij}=f_{\delta}$ with boundary value $w=\psi$ where $\psi\in W^{2, q}(\Omega)$, we conclude that 
 $w\in W^{2, q}(\Omega)$ and therefore $u\in W^{4,q}(\Omega)$ with universal estimate
 \begin{equation*}
   \|u\|_{W^{4,q}(\Omega)}\leq C_6.
   \end{equation*}
\end{proof}

\section{Proofs of the main theorems}
\label{pf_sec}
In this section, we prove Theorems \ref{SBV1}, \ref{SBV2}, \ref{SBV3} and \ref{SBV4}.
\begin{proof}[Proof of Theorem \ref{SBV1}]
The existence and uniqueness result in (ii) follows from the existence in (i) and the uniqueness result in Lemma \ref{uni_lem}.
It remains to prove (i). 

The proof of (i) uses the a priori estimates in Theorem \ref{thmw4p} and degree theory as in \cite{CW, TW05} (see also \cite{LMT}). Since the proof is short, we include it here. Assume $q>n$.

Fix $\alpha \in (0,1)$. For  a large constant $R>1$ to be determined, define a bounded set $D(R)$ in $C^{\alpha}(\ov{\Omega})$ as follows:
$$D(R) = \{ v \in C^{\alpha}(\ov{\Omega}) \ | \ v \ge R^{-1}, \ \| v\|_{C^{\alpha}(\ov{\Omega})} \le R\}.$$

For $t \in [0,1]$, we will define an operator $\Phi_t : D(R) \rightarrow C^{\alpha}(\ov{\Omega})$ as follows.
Given $w \in D(R)$, define $u \in C^{2, \alpha}(\ov{\Omega})$ to be the unique uniformly convex solution to
\begin{equation}\label{LS1}
 \left\{
 \begin{alignedat}{2}
   \det D^2 u ~&=w^{\frac{1}{\theta-1}} \h~&&\text{in} ~\Omega, \\\
 u&= \varphi \h~&&\text{on}~ \p\Omega.
 \end{alignedat}
 \right.
\end{equation}
The existence of $u$ follows from the boundary regularity result of the Monge-Amp\`ere equation established by Trudinger and Wang \cite{TW08}.
Next, let $w_t \in W^{2,q}(\Omega)$  be the unique solution to the equation
\begin{equation}\label{LS2}
 \left\{
 \begin{alignedat}{2}
   U^{ij} (w_t)_{ij} ~&=tf_{\delta}(\cdot, u, Du, D^2 u) \h~&&\text{in} ~\Omega, \\\
 w_t&= t\psi + (1-t) \h~&&\text{on}~ \p\Omega.
 \end{alignedat}
 \right.
\end{equation}
Because $q>n$, $w_t$ lies in  $C^{ \alpha}(\ov{\Omega})$.  We define $\Phi_t$ to be the map sending $w$ to $w_t$.

We note that:
\begin{enumerate}
\item[(i)] $\Phi_0(D(R)) = \{ 1\} $, and in particular, $\Phi_0$ has a unique fixed point.
\item[(ii)] The map $[0,1] \times D(R) \rightarrow C^{\alpha}(\ov{\Omega})$ given by $(t,w) \mapsto \Phi_t(w)$ is continuous.  
\item[(iii)] $\Phi_t$ is compact for each $t \in [0,1]$.
\item[(iv)] For every $t\in [0,1]$, if $w \in D(R)$ is a fixed point of $\Phi_t$ then $w \notin \partial D(R)$.
\end{enumerate}
Indeed, part (iii) follows from the standard \emph{a priori} estimates for the two separate equations (\ref{LS1}) and (\ref{LS2}).  For part (iv), let $w>0$ be a fixed point of $\Phi_t$.  
Then $w \in W^{2,q}(\Omega)$ and hence $u \in W^{4,q}(\Omega)$.  Next we apply Theorem \ref{thmw4p} to obtain $w>R^{-1}$ and $\| w \|_{C^{\alpha}(\ov{\Omega})} < R$ for some $R$ sufficiently large, depending only on the initial data
but independent of $t\in [0,1]$.

Then the Leray-Schauder degree of $\Phi_t$ is well-defined for each $t$ and is constant on $[0,1]$ (see \cite[Theorem 2.2.4]{OCC}, for example).  $\Phi_0$ has a fixed point and hence $\Phi_1$ 
must also have a fixed point $w$, giving rise to a uniformly convex solution $u\in W^{4,q}(\Omega)$ of  our second boundary value problem (\ref{Abreu2})-(\ref{fdel}).
\end{proof}

\begin{proof}[Proof of Theorem \ref{SBV2}] 
If $\bar c_0=\bar C_\ast =0$, then
the existence of a unique uniformly convex solution $u_\e \in W^{4,q}(\Omega)$ to the system (\ref{Abreu2e})-(\ref{fdele}) for all $q\in (n,\infty)$ follows from Theorem \ref{SBV1} (ii).
If $\rho$ is sufficiently large (depending only on $\bar c_0 +\bar C_\ast$, $\Omega_0$ and $\Omega$),
then the existence of a uniformly convex solution $u_\e \in W^{4,q}(\Omega)$ to the system (\ref{Abreu2e})-(\ref{fdele}) for all $q\in (n,\infty)$ follows from Theorem \ref{SBV1} (i); moreover, since we are interested in the limit of $\{u_\e\}$ when $\e\rightarrow 0$, we can assume that $\e$
is sufficiently small  (depending only on $\bar c_0 +\bar C_\ast$, $\Omega_0$ and $\Omega$) so that Lemma \ref{lem_uep} applies. 

In all cases, by Lemma \ref{lem_uep}, there is a universal constant $C$ independent of $\e$ such that
\begin{equation} 
\label{Jepconvex}
 \int_{\partial \Omega} \e (u_\e)_{\nu}^2  + \rho \int_{\Omega_0}|u_\e-\varphi|^2 dx +    \int_{\Omega\setminus\Omega_0} \frac{1}{\e}|u_\e-\varphi|^2 dx \leq C.
\end{equation} 
{\it Step 1: A subsequence of $\{u_\e\}$ converges.} First, we show that, up to extraction of a subsequence,  $u_\e$ converges uniformly on compact subsets of $\Omega$ to a convex function $u\in \bar S[\varphi, \Omega_0]$ where $\bar S[\varphi, \Omega_0]$ is defined as in (\ref{barS}). 
Indeed, by  Lemma \ref{ucon_lem} (i),  $u_\e=\varphi $ on $\p\Omega$ and (\ref{Jepconvex}) where $\rho>0$, we have
\begin{equation*}
\|u_\e\|^2_{L^{\infty}(\Omega)}  \leq C(n,\Omega, \max_{\p \Omega} u_\e) + C(n,\Omega)\int_{\Omega} |u_\e|^2 dx \leq C
\end{equation*}
for a universal constant $C>0$.
It follows that the sequence $\{u_\e\}$ is uniformly bounded. By Lemma \ref{ucon_lem} (ii), $|Du_\e|$ is uniformly bounded on compact subsets of $\Omega$. Thus, by the Arzela--Ascoli  theorem, up to extraction of a subsequence,  $u_\e$ converges uniformly on compact subsets of $\Omega$ to a convex function $u$. 
Moreover, we can assume that $u_\e$ converges to $u$ on $W^{1,2}(\Omega_0)$.
Using (\ref{Jepconvex}), we get $u\in \bar S[\varphi, \Omega_0]$.

Let $G(t)=\frac{t^{\theta}-1}{\theta}$ for $t>0$ (when $\theta=0$, we set $G(t)=\log t$).

Next, consider the following functional $J_\e$ over the set of convex functions $v$ on $\overline{\Omega}$:
\begin{equation}
\label{Je_def}
J_{\e}(v)=\int_{\Omega_0}  [F^0(x, v(x)) + F^1(x, Dv(x))]dx +\frac{1}{2\e}\int_{\Omega\setminus\Omega_0} (v-\varphi)^2 dx-\e\int_{\Omega} G( \det D^2 v) dx.
\end{equation}
By the Rademacher theorem (see \cite[Theorem 2, p.81]{EG}), $v$ is differentiable a.e. By the Alexandrov theorem (see \cite[Theorem 1, p.242]{EG}), $v$ is twice differentiable a.e and at those points of twice differentiability, we denote, with a slight abuse of notation, $D^2 v$
its Hessian matrix. Thus, the functional $J_\e$ is well defined with this convention.

Let $U^{\nu\nu}_\e=U^{ij}_\e \nu_i\nu_j$ be as in the proof of Lemma \ref{ubound}.  Let $\tau$ be the tangential direction along $\p\Omega$. Let $K$ be the curvature of $\p\Omega$. 
Since $u_\e=\varphi$ on $\p\Omega$,
in two dimensions, we have as in (\ref{uij1})
\begin{equation}
\label{ueij}
U_\e^{\nu\nu}= (u_\e)_{\tau\tau}=K(u_\e)_\nu-K\varphi_\nu + \varphi_{\tau \tau}.
\end{equation}
{\it Step 2: Almost minimality property of $u_\e$.} We show that if $v$ is a convex function in $\overline{\Omega}$ with $v=\varphi$ in a neighborhood of $\p\Omega$ then
\begin{equation}
\label{Je_min}
J_\e(v)-J_\e(u_\e)\geq  \e \int_{\p\Omega} \psi U_{\e}^{\nu\nu} \p_{\nu}(u_\e-\varphi) +  \int_{\p\Omega_0}(v-u_\e)\nabla_p F^1(x, Du_\e(x)) \cdot \nu_0 dS.
\end{equation} 
The proof of (\ref{Je_min}) uses mollification to deal with general convex functions $v$. For $h>0$, let $$\Omega_h = \{x\in\Omega\mid \dist(x,\p\Omega)>h\}~\text{and }
v_h(x)=h^{-n} \int_{\Omega}\phi(\frac{x-y}{h}) v(y)dy~\text{for } x\in\Omega_h$$ 
where
$\phi\geq 0, ~\phi\in C^{\infty}_0(\R^n), ~\text{supp }\phi\subset B_1(0) ~\text{and }\int_{\R^n} \phi dx=1.$
Clearly, $v_h\rightarrow v$ uniformly on compact subsets of $\Omega$. Since $v=\varphi$ near $\p\Omega$ and $\varphi\in C^{3,1}(\overline{\Omega})$ is uniformly convex in $\overline{\Omega}$, we can extend $v_h$ to be a uniformly convex $C^{3}(\overline{\Omega})$ function, still denoted by $v_h$, such that
\begin{equation}
\label{Dkv}
D^k v_h\rightarrow D^k v \text{ in a 
neighborhood of } \p\Omega~\text{for all } k\leq 2.
\end{equation}
By \cite[Lemma 6.3]{TW00}, we have
$$\lim_{h\rightarrow 0} \int_{\Omega_h} G(\det D^2 v_h) =\int_{\Omega} G(\det D^2 v).$$
This together with (\ref{Dkv}) implies that
\begin{equation}
\label{Jevh_lim}
\lim_{h\rightarrow 0} J_\e(v_h)= J_\e (v).
\end{equation}
Now, we estimate $J_\e(v_h)-J_\e(u_\e)$. 

As in the proof of Lemma \ref{ubound}, we have
$$-G(\det D^2 v_h) + G(\det D^2 u_\e)\geq w_\e U^{ij}_\e (u_\e-v_h)_{ij}.$$
Integrating by parts twice, recalling  $U^{ij}_\e (w_\e)_{ij} =\e^{-1} f_{\e}(\cdot, u_\e, Du_\e, D^2 u_\e),$ and (\ref{nunu}),
 we get 
\begin{multline}
\label{vh_G}
\int_{\Omega}(-G(\det D^2 v_h) + G(\det D^2 u_\e)) dx\geq \int_{\Omega}w_\e U^{ij}_\e (u_\e-v_h)_{ij}\\= \int_{\Omega} \e^{-1}f_{\e}(\cdot, u_\e, Du_\e, D^2 u_\e) (u_\e-v_h) dx- \int_{\p\Omega} (w_\e)_i U^{ij}_\e (u_\e-v_h) \nu_j+ \int_{\p\Omega} \psi U_{\e}^{\nu\nu} \p_{\nu}(u_\e-v_h).
\end{multline}
From the convexity of $F^0, F^1$ and $(v-\varphi)^2$, 
we have
\begin{eqnarray}
\label{Jecon_vh}
J_\e(v_h)-J_\e(u_\e)&\geq& \int_{\Omega_0} [f^0(x, u_\e(x)) (v_h-u_\e) + \nabla_p F^1(x, Du_\e(x)) (D v_h-Du_\e)] dx \nonumber\\
&+& \frac{1}{\e}\int_{\Omega\setminus\Omega_0} (u_\e-\varphi) (v_h-u_\e) dx +\e \int_{\Omega}(-G(\det D^2 v_h) + G(\det D^2 u_\e)) dx.
\end{eqnarray}
In view of (\ref{fdele}) and (\ref{vh_G}), we can integrate by parts the right hand side of (\ref{Jecon_vh}) to get, after a simple cancellation,
\begin{multline}
\label{Je_vh}
J_\e(v_h)-J_\e(u_\e)\geq 
- \e\int_{\p\Omega} (w_\e)_i U^{ij}_\e (u_\e-v_h) \nu_j+ \e \int_{\p\Omega} \psi U_{\e}^{\nu\nu} \p_{\nu}(u_\e-v_h)\\ +  \int_{\p\Omega_0}(v_h-u_\e)\nabla_p F^1(x, Du_\e(x)) \cdot \nu_0.
\end{multline}
By (\ref{Dkv}), the right hand side of (\ref{Je_vh}) tends to the right hand side of (\ref{Je_min}) when $h\rightarrow 0$. 
On the other hand, in view of (\ref{Jevh_lim}), the left hand side of (\ref{Je_vh}) tends to the left hand side of (\ref{Je_min}) when $h\rightarrow 0$.
Therefore, (\ref{Je_min}) is proved by letting $h\rightarrow 0$ in (\ref{Je_vh}).

{\it Step 3: Minimality of $u$.} We show that $u$ (in Step 1) is a minimizer of the functional $J$ defined by (\ref{J_def}) over $\bar S[\varphi, \Omega_0]$.
For all $v\in\bar S[\varphi,\Omega_0]$ (extended by $\varphi$ on $\Omega\setminus\Omega_0$), we use (\ref{Je_min})  to conclude that
$$J_\e(v_\e)-J_\e(u_\e)\geq \e \int_{\p\Omega} \psi U_{\e}^{\nu\nu} \p_{\nu}(u_\e-\varphi)-O(\e)$$
where
$$v_\e=(1-\e) v + \e\varphi \in \bar S[\varphi, \Omega_0].$$
Since $\lim_{\e\rightarrow 0} J(v_\e)=J(v)$, it follows that
\begin{equation}
\label{Jv}
J(v)\geq \liminf_{\e} J(u_\e) + \liminf_{\e} \e[\int_{\Omega} G(\det D^2 v_\e)- G(\det D^2 u_\e)] + \liminf_{\e}\e \int_{\p\Omega} \psi U_{\e}^{\nu\nu} \p_{\nu}(u_\e-\varphi).
\end{equation}
From (\ref{Jepconvex}), we have
$$\int_{\p\Omega} |(u_\e)_\nu| \leq C\e^{-1/2}$$
and hence, invoking (\ref{ueij}), one finds that
\begin{equation}
\label{ue_bdr}
 \e\int_{\p\Omega} \psi U_{\e}^{\nu\nu} \p_{\nu}(u_\e-\varphi) \geq -C\e \int_{\p\Omega}[1 +|(u_\e)_\nu|]\geq -C\e^{1/2}.
\end{equation}
Observe from $0\leq \theta<1/2$ that,  $G(d) \leq C(1 + d^{1/2})~\text{for all } d>0.$
It follows that
\begin{eqnarray}
\int_{\Omega}G(\det D^2 u_\e)&\leq& C\int_{\Omega} (1+ (\det D^2 u_\e)^{1/2}) dx \leq C\int_{\Omega} (1 + \Delta u_\e) dx\nonumber\\ &=&
C(|\Omega| + \int_{\p\Omega} (u_\e)_\nu) \leq C(1 +\e^{-1/2}).
\label{ue_det}
\end{eqnarray}
Note that $D^2 v_\e\geq \e D^2\varphi$. Therefore $\det D^2 v_\e \geq \e^2 \det D^2\varphi\geq C_1\e^2$ in $\Omega$  for $C_1>0 $ and hence
\begin{equation}
\label{ve_det}
\e\int_{\Omega}G(\det D^2 v_\e) \geq\e\int_{\Omega} G(C_1 \e^2 ) \rightarrow 0~\text{when } \e\rightarrow 0.
\end{equation}
From (\ref{Jv})--(\ref{ve_det}), we easily obtain
\begin{equation}
\label{Jev_ineq}
J(v)\geq \liminf_{\e}J(u_\e).
\end{equation}
Since $u_\e$ converges uniformly to $u$ on $\overline{\Omega_0}$, by Fatou's lemma, we have
\begin{equation}
\label{Jev_ineq2}
\liminf_{\e}\int_{\Omega_0}F^0(x, u_\e(x)) dx\geq \int_{\Omega_0} F^0(x, u(x)) dx.
\end{equation}
 From the convexity of $F^1(x, p)$ in $p$ and the fact that $u_\e$ converges to $u$ on $W^{1,2}(\Omega_0)$, by lower semicontinuity, we have
 \begin{equation}
 \label{Jev_ineq3}
 \liminf_{\e}\int_{\Omega_0}F^1(x, Du_\e(x)) dx\geq \int_{\Omega_0} F^1(x, Du(x)) dx.
 \end{equation}
 Therefore, by combining (\ref{Jev_ineq})--(\ref{Jev_ineq3}), we obtain
 $$J(v)\geq \liminf_{\e}J(u_\e)\geq J(u)~\text{for all } v\in\bar S[\varphi,\Omega_0],$$
 showing that $u$ is a minimizer of the functional $J$ defined by (\ref{J_def}) over $\bar S[\varphi, \Omega_0]$.

 {\it Step 4: Full convergence of $u_\e$ to the unique minimizer of $J$.} Since $\rho>0$, functional $J$ defined by (\ref{J_def}) over $\bar S[\varphi, \Omega_0]$ has a unique minimizer in $\bar S[\varphi, \Omega_0]$. Thus, Steps 1 and 3 actually shows that the whole sequence
 $\{u_\e\}$ converges to the unique minimizer of $J$.
The proof of the theorem is completed.
\end{proof}

\begin{proof}[Proof of Theorem \ref{SBV3}]
When $H(d)= d^{\theta-1}$ where $0\leq \theta<\frac{1}{2}$, the proof of the uniqueness of solutions in (i) is similar to that of Lemma \ref{uni_lem} so we omit it.
 The existence proof uses a priori estimates and degree theory as in Theorem \ref{SBV1}. Here, 
 we only focus on proving the a priori estimates. 
 The key is to obtain the positive bound from below and above for $\det D^2 u$:
 \begin{equation}
 \label{detD2u}
 C^{-1}\leq \det D^2 u\leq C~\text{in}~\Omega.
 \end{equation}
  Once (\ref{detD2u}) is established, 
 we can apply the global $W^{2, 1+\e_0}(\Omega)$ estimates in Theorem \ref{global-w21} for $u$ and argue as in the proof of Theorem \ref{thmw4p} that $u\in C^{2,\alpha}(\overline{\Omega})$. A bootstrap argument concludes the proof.
 
 It remains to prove (\ref{detD2u}).\\
 (i) First, by the convexity of $F$ and $u$, we have $$U^{ij} w_{ij}=-\div (\nabla_p F(Du))=-F_{p_i p_j}(Du) u_{ij} =-\text{trace}(D^2 F(D u) D^2 u)\leq 0~\text{in}~\Omega.$$ By the maximum principle, the function $w$ attains its minimum value on $\p\Omega$. It follows that $$w\geq \inf_{\p\Omega}\psi:= C>0~\text{in }\Omega.$$ 
 From the assumptions on $H$ and $w=H(\det D^2 u)$, we deduce that
 $$\det D^2 u\leq C<\infty~\text{in }\Omega.$$
From $\det D^2 u\leq C$ in $\Omega$, $u=\varphi$ on $\p\Omega$, and the uniform convexity of $\Omega$, we can construct an explicit barrier to show that
 $|Du|\leq C$ in $\Omega$ for a universal constant $C$. This together with the boundedness assumption on $F_{p_i p_j}(p)$ gives $$F_{p_i p_j}(Du(x))\leq C_1 I_2~\text{in }\Omega.$$ We compute, in $\Omega$,
$$U^{ij} (w + C_1|x|^2)_{ij}=-F_{p_i p_j}(Du) u_{ij} + 2 C_1 \text{trace } (U^{ij})\geq -C_1 \trace(u_{ij}) + 2C_1\Delta u= C_1\Delta u\geq 0.$$
By the maximum principle, $w(x) + C_1|x|^2$ attains it maximum value on the boundary $\p\Omega$. Recall that $w=\psi$ on $\p\Omega$.
Thus, for all $x\in\Omega$, we have
\begin{equation}
\label{wmaxi}
w(x) \leq w(x) + C_1|x|^2 \leq \max_{\p\Omega} (\psi + C_1|x^2|)\leq C.
\end{equation}
 From this universal upper bound for $w$, we can use $w=H(\det D^2 u)$  and the assumptions on $H$ to obtain $\det D^2 u\geq C^{-1}>0$ in $\Omega$. Therefore, (\ref{detD2u}) is proved.\\
(ii) Assume $0\leq F_{p_i p_j}(p)\leq C_\ast I_2$. As above, using  the convexity of $F$ and $u$, we can prove that $w\geq C>0$ in $\Omega$ for a universal constant $C>0$. From 
$$U^{ij} (w + C_\ast |x|^2)_{ij}=-F_{p_i p_j}(Du) u_{ij} + 2 C_\ast \text{trace } (U^{ij})\geq -C_\ast \trace(u_{ij}) + 2C_\ast\Delta u= C_\ast\Delta u \geq 0~\text{in}~\Omega$$
and the maximum principle, we also have, as in (\ref{wmaxi}),  $w\leq C_1$ in $\Omega$. 
Consequently, $$0<C\leq w \leq C_1<\infty \text{ in }\Omega.$$ From $w=H(\det D^2 u)$ and the fact that $H^{-1}$ maps compact subsets of $(0, \infty)$ into compact subset of $(0,\infty)$, we find that
$C_1\leq \det D^2 u\leq C_2$ in $\Omega$. Therefore, (\ref{detD2u}) is also proved.
\end{proof}
\begin{proof}[Proof of Theorem \ref{SBV4}]
The proof is very similar to that of Theorem \ref{SBV1}. We focus here on the a priori estimates. The key point is to establish the universal bound for $u$ as in Lemma \ref{ubound}. We do this via proving the estimate of the type
(\ref{ubound_red}). We use the same notation as in the proof of Lemma \ref{ubound} with $\theta=0$. The estimate (\ref{concave2}) with $n=2$ now becomes
\begin{equation}
\label{concave4}
 \int_{\partial \Omega} K \psi u_{\nu}^2 
 \leq  C + C\left(
\int_{\p \Omega}  (u_{\nu}^+)^2 \right)^{1)/2} +  \int_{\Omega} [-(u^3-u) + \Delta u] (u-\tilde u)dx.
\end{equation}
Since $\tilde u$ is universally bounded, there is a universal constant $C>0$ such that 
$$-(u^3-u)(u-\tilde u)\leq C.$$
Moreover, from $u\leq \sup_{\p\Omega}\varphi\leq C$ by the convexity of $u$, we have
$$\int_{\Omega} \Delta u (u-\tilde u)\leq C\int_{\p\Omega} \Delta u dx = C\int_{\p\Omega} u_{\nu}.$$
Thus (\ref{concave4})  gives
\begin{equation*}
 \int_{\partial \Omega} K \psi u_{\nu}^2 
 \leq  C + C\left(
\int_{\p \Omega}  (u_{\nu}^+)^2 \right)^{1)/2} +  C\int_{\p\Omega} u_{\nu}.
\end{equation*}
A simple application of Young's inequality to the above inequality together with the fact that $\inf_{\p\Omega} (K\psi)>0$ shows that $\int_{\p\Omega} u_\nu^2 \leq C$ which is exactly what we need to prove.
\end{proof}

\end{document}